\documentclass[10pt]{article}
\usepackage{amsfonts,amsthm}
\usepackage{bm}
\usepackage{mathrsfs,bbm}
\usepackage{amsmath,amssymb}
\usepackage{graphicx,graphics}
\usepackage{subfigure}
\usepackage{slashbox}
\usepackage{color}
\usepackage{multirow}

\newcommand{\jbhalf}{j\!-\!\frac12}
\newcommand{\jfhalf}{j\!+\!\frac12}

\newcommand{\dt} {\tau}

\newcommand{\Lcal}{\mathcal{L}}
\newcommand{\Kcal}{\mathcal{K}}

\newcommand{\Tcal} {\mathcal{T}}
\newcommand{\Ocal} {\mathcal{O}}

\newcommand{\Ht}{H^{1}(\Tcal_h)}

\newcommand{\nl}{{n,\ell}}

\newcommand{\norm}   [1] {\Vert#1\Vert}

\newcommand{\jump}   [1] {[\![#1]\!]}

\newcommand{\be}{\begin{equation}}
\newcommand{\ee}{\end{equation}}
\newcommand{\bes}{\begin{subequations}}
\newcommand{\ees}{\end{subequations}}

\newtheorem  {theorem}      {\noindent Theorem}[section]
\newtheorem  {lemma}        {\noindent Lemma}[section]

\theoremstyle{remark}
\newtheorem  {Remark}        {\noindent Remark}[section]

\title{Local discontinuous Galerkin methods with explicit-implicit-null
time discretizations for solving nonlinear diffusion problems}
\author{
Haijin Wang\footnotemark[2]
\and
Qiang Zhang
\footnotemark[3]
\and
Shiping Wang\footnotemark[4]
\and
Chi-Wang Shu\footnotemark[5]
}

\begin{document}
\maketitle
\renewcommand{\thefootnote}{\fnsymbol{footnote}}

\footnotetext[2]{School of Science, Nanjing University of
Posts and Telecommunications,
Nanjing 210023, Jiangsu Province, P.~R.~China.
E-mail: hjwang@njupt.edu.cn.
Research sponsored by NSFC grants 11601241 and 11671199,
Natural Science Foundation of Jiangsu Province grant BK20160877.}
\footnotetext[3]{Department of Mathematics, Nanjing University,
Nanjing 210093, Jiangsu Province, P.~R.~China.
E-mail: qzh@nju.edu.cn.
Research supported by NSFC grants  11671199 and 11571290.}
\footnotetext[4]{College of Shipbuilding Engineering,
Harbin Engineering University, Harbin 150001, P. R. China.
E-mail: wangshiping@hrbeu.edu.cn. Research supported by NSFC grant 11672082.}
\footnotetext[5]{Division of Applied Mathematics, Brown University,
Providence, RI 02912, U.S.A.  E-mail: shu@dam.brown.edu.
Research supported by ARO grant W911NF-15-1-0226 and NSF grant DMS-1719410.}

\begin{abstract}
In this paper we discuss
the local discontinuous Galerkin methods
coupled with two specific explicit-implicit-null
time discretizations for solving one-dimensional nonlinear diffusion
problems $U_t=(a(U)U_x)_x$.
The basic idea is to add and subtract two equal terms $a_0 U_{xx}$
on the right hand side of the partial differential equation,
then to treat the term $a_0 U_{xx}$ implicitly and the other terms
$(a(U)U_x)_x-a_0 U_{xx}$ explicitly.
We give stability analysis for the method on a simplified model
by the aid of energy analysis, which gives a guidance for the
choice of $a_0$,
i.e, $a_0 \ge \max\{a(u)\}/2$
%and $a_0 > \max\{a(u)\}/2$
to ensure the
unconditional stability of the first order and second order schemes.
The optimal error estimate is also derived for the simplified
model, and
numerical experiments are given to demonstrate the stability,
accuracy and performance of the schemes for nonlinear diffusion
equations.
\end{abstract}



\section{Introduction}
\label{sec1}

Many partial differential equations (PDE) which arise
in physics or engineering involve the computation of nonlinear diffusion,
such as the miscible displacement in porous media \cite{Ewing} which is widely used in the
exploration of underground water, oil, and gas,
the carburizing model \cite{Carburizing} which is derived in the
chemical heat treatment in mechanical industry,
the high-field model in semiconductor device simulations \cite{HF,HF1},
and so on. It is well known that the time discretization is a very important issue
for such problems containing complicated nonlinear diffusion coefficients.
Explicit time marching always suffer from stringent time step
restriction.
Implicit time marching can overcome the constraint of small time step,
however, this method becomes cumbersome if the diffusion coefficients
vary in space or depend on the solution (quasi-linear or nonlinear cases),
since a Newton iteration is required at each time step.

To cope with both the shortcomings of the explicit and implicit time marching methods,
we notice that the implicit time discretization can be actually
very efficient
for solving diffusion equations with constant coefficients, since the inverse
matrix is only needed to be solved once. This observation inspire us to add and subtract a term with
constant diffusion coefficient $a_0 U_{xx}$ on the right hand side of the considered PDE
\be \label{eq:diff}
U_t=(a(U)U_x)_x,  \qquad x \in \Omega=[a,b], \, t\in (0,T]
\ee
where $a(U)\ge 0$ and $a(U)$ is bounded and smooth,
and then apply the implicit-explicit (IMEX) time marching methods \cite{Ascher}
to the equivalent PDE
\be \label{eq:diff:2}
U_t=\underbrace{(a(U)U_x)_x-a_0 U_{xx}}_{T_1} + \underbrace{a_0 U_{xx}}_{T_2}.
\ee
Namely, we treat the damping term $T_2$ implicitly and the remaining term $T_1$ explicitly.

Such idea had been adopted by Douglas and Dupont \cite{Douglas}
 to assure the stability for a nonlinear diffusion equation on
a rectangle. The similar idea has also been adopted,
for example, by Smereka \cite{Smereka} in the context of flow by mean curvature and surface diffusion, by Jin
and Filbet \cite{Filbet} in the context of the Boltzmann equation of rarefied gas dynamics when the Knudsen
number is very small,  in the context of hyperbolic systems with diffusive relaxation \cite{Boscarino},
and for the solution of PDEs on surfaces \cite{Macdonald}.
In a recent study, Duchemin and Eggers \cite{Duchemin} proposed to call this method as
explicit-implicit-null (EIN) method.

In this paper, we exploit  EIN method coupled with local discontinuous Galerkin (LDG)
spatial discretization to solve the nonlinear diffusion equation (\ref{eq:diff}).
The LDG method was introduced by Cockburn and Shu in  \cite{Cock:Shu:LDG}
for solving convection diffusion equations, motivated by the work
of Bassi and Rebay  \cite{Bassi} for the compressible Navier-Stokes equations.
The idea of the LDG method is to rewrite the equations with higher order derivatives
into an equivalent first order system, then apply the DG method \cite{Cock:Shu:DG} to the system,
so the LDG scheme shares the advantages of the DG methods.
It can easily handle meshes with hanging nodes, elements of general shapes and local
spaces of different types, thus it is flexible for $hp$-adaptivity.
Besides, a key advantage of the LDG scheme is the local solvability,
that is, the auxiliary variables approximating the derivatives of
the solution can be locally eliminated \cite{Cock:Shu:LDG,Castillo}.

Two EIN time marching schemes with LDG spatial discretization (EIN-LDG)
will be analyzed in the present paper.
The first order scheme is a combination of forward Euler discretization
and backward Euler discretization for the explicit part and
the implicit part, respectively, which was considered in our previous work
\cite{WSZ:1,WSZ:2} for solving one-dimensional convection-diffusion problem
and time-dependent fourth order problem.
The second order scheme to be considered in this paper is different from
the one we used in \cite{WSZ:1,WSZ:2}, the new scheme is a modification of
the second order scheme proposed by Cooper and Sayfy \cite{Cooper}.
By the aid of the energy analysis, we show that the proposed schemes
are unconditionally stable provided $a_0 \ge a/2$
%( the first order scheme)
%and $a_0 >a/2$ (the second order scheme)
for the simplified linear
model $U_t=a U_{xx}$, where $a>0$ is a constant.
The optimal error estimates will also be given by energy analysis
for the simplified model.
%\textcolor[rgb]{0,0,1}
{We would like to point out that it is necessary to do  energy analysis
even for the linear model, since the spatial discretization may result in
non-normal systems with a growing dimension, hence
the spectral stability analysis based on scalar eigenvalues arguments
may be misleading \cite{Levy}.}

Based on the stability and error analysis for the simplified model, we
propose a guidance for the choice of $a_0$ for the
general model $U_t=(a(U)U_x)_x$,
that is, $a_0  \ge \max\{a(u)\}/2$, where $u$ is the numerical solution.
It is worth pointing out that it is not necessary to scan the maximum of
$a(u)$ and adjust $a_0$ at every time level, theoretically we can choose $a_0$ as a
sufficiently large positive constant. However, too large $a_0$ may
cause larger errors and may require a smaller time step from our numerical observation.
So in practical computing, we adjust $a_0$ after certain number of time steps
to alleviate numerical errors and to keep high efficiency in the meantime.
We point out that the EIN-LDG schemes
also work well for convection-diffusion problems with nonlinear diffusions.
To verify the accuracy and performance of the proposed schemes, we
present several numerical experiments, including the simulations
for porous media equations and the high-field model in semiconductor
device simulations.

The  paper is organized as follows.
In Section \ref{sec2} we present the semi-discrete LDG scheme and
the  time-discretization methods.
Sections \ref{sec3} and \ref{sec4} are devoted to the stability and error
analysis of the EIN-LDG methods, respectively.
In Section \ref{sec5} we will present numerical results to verify
the accuracy and the performance of the proposed schemes.
The conclusion is given in Section \ref{sec6}.

\section{The LDG scheme and time-discretization}
\label{sec2}
In this section, we will present the discontinuous finite element space,
the semi-discrete LDG scheme, and the  implicit-explicit (IMEX) Runge-Kutta (RK)
time-discretization methods.
\setcounter{equation}{0}
\subsection{The discontinuous finite element space}
Let $\Tcal_h=\{I_j=(x_{\jbhalf},x_{\jfhalf})\}_{j=1}^N$ be a partition of
$\Omega$, where $x_{\frac12}=a$ and $x_{N\!+\!\frac12}=b$ are
the two boundary endpoints.
Denote the cell length as $h_j=x_{\jfhalf}-x_{\jbhalf}$ for
$j=1,\ldots ,N$, and define $h=\max_j h_j$.
We assume $\Tcal_h$ is quasi-uniform in this paper, that is,
there exists a positive
constant $\rho$ such that for all $j$ there holds $h_j/h \ge \rho$, as $h$
goes to zero.

Associated with this mesh, we define the discontinuous finite element space
\begin{equation}
V_h = \big\{\,v\in L^2(\Omega):v|_{I_j}
\in \mathcal{P}_k(I_j),\,\forall j=1,\ldots,N\,\big\}\, ,
\end{equation}
where $\mathcal{P}_k(I_j)$ denotes the space of polynomials in
$I_j$ of degree at most $k$.
Note that functions in this space are allowed to have
discontinuities across element interfaces.
At each element interface point, for any piecewise function $p$,
there are two traces along
the right-hand and left-hand, denoted by $p^{+}$ and $p^{-}$, respectively.
 The jump is denoted by $\jump{p} = p^{+} - p^{-}$.

\subsection{The semi-discrete LDG scheme}

We begin with equation (\ref{eq:diff:2}) to define the LDG scheme.
Denote by $b(U)=\sqrt{a(U)}$, by introducing $P=b(U)U_x$ and $Q=U_x$,
the equation can be written as
\begin{subequations}
\begin{align}
& U_t +(a_0 Q-b(U)P)_x = a_0 Q_x,\\
& P - B(U)_x = 0, \\
& Q -  U_x = 0,
\end{align}
\end{subequations}
where $B(U)=\int^U b(s)\mathrm{d}s$.
The semi-discrete LDG scheme is to find $u,q,p\in V_h$,
such that for arbitrary $v,r,w\in V_h$ we have
\begin{subequations} \label{var:nonlinear}
\begin{align}
(u_t,v)=&\, \tilde{\Lcal}(b(u)p,v) - a_0 \Lcal(q,v) + a_0 \Lcal(q,v),
\label{var:nonlinear:1} \\
(q,r) = &\, \Kcal(u,r), \\
(p,w) = &\, \tilde{\Kcal}(B(u),w),
\end{align}
\end{subequations}
where
\begin{subequations}
\begin{align}
\Lcal(q,v) = &\,-\sum_{j=1}^N \left[(q,v_x)_j -\hat{q}_{\jfhalf}
v_{\jfhalf}^- + \hat{q}_{\jbhalf} v_{\jbhalf}^+\right],
\\
\Kcal(u,r) =&\,-\sum_{j=1}^N\left[(u,r_x)_j -\hat{u}_{\jfhalf}
r_{\jfhalf}^- + \hat{u}_{\jbhalf} r_{\jbhalf}^+\right],
\\
\tilde{\Lcal}(b(u)p,v) = &\,-\sum_{j=1}^N\left[(b(u)p,v_x)_j -(\widehat{b(u)}\hat{p})_{\jfhalf}
v_{\jfhalf}^- + (\widehat{b(u)}\hat{p})_{\jbhalf} v_{\jbhalf}^+\right],
\\
\tilde{\Kcal}(B(u),w) =&\,-\sum_{j=1}^N\left[(B(u),w_x)_j -\widehat{B(u)}_{\jfhalf}
w_{\jfhalf}^- + \widehat{B(u)}_{\jbhalf} w_{\jbhalf}^+\right].
\end{align}
\end{subequations}
The ``hat'' terms are numerical fluxes which are taken as in \cite{Cock:Shu:LDG,Yan,Xu}
\[ \hat{q} = q^+, \quad
 \hat{p} = p^{+}, \quad
 \hat{u} = u^{-}, \quad
\widehat{B(u)} = B(u^-)\]
and
\[\widehat{b(u)} =\left\{ \begin{array}{cc}\jump{B(u)}/\jump{u} & \mbox{if} \quad \jump{u} \neq 0\\
                                 b((u^++u^-)/2)  & \mbox{otherwise}
               \end{array} \right. ,\]
%\textcolor[rgb]{0,0,1}
{where we omitted the subscripts $\jbhalf$ and $\jfhalf$.
For simplicity of analysis, we
consider the periodic boundary conditions, i.e, $w_{\frac12}^-=w_{N+\frac12}^-$
and $w_{N+\frac12}^+=w_{\frac12}^+$ for $w=u,p,q$. For other boundary conditions,
such as Dirichlet boundary condition problems, we refer the readers to \cite{Castillo,WSZ:4}
for the setting of numerical fluxes.}

The initial solution $u^0$ can be taken as any approximation of the initial condition $U(x,0)$,
for example the Gauss-Radau projection of $U(x,0)$.

We have the following lemma which can be obtained easily by integrating by parts,
so we omit the  proof and refer the reader to \cite{Zhang:Gao}.
\begin{lemma}
For any pairs of $(u_1,q_1)$ and $(u_2,q_2)$ belonging to $V_h \times V_h$, we have
\begin{equation} \label{property}
\Lcal(q_1,u_2) = -\Kcal(u_2,q_1) = -(q_2,q_1),
\end{equation}
and for any pairs of $(u_1,p_1)$ and $(u_2,p_2)$ belonging to $V_h \times V_h$, we have
\begin{equation} \label{property:1}
\tilde{\Lcal}(b(u_1)p_2,u_1) = -\tilde{\Kcal}(B(u_1),p_2) = -(p_1,p_2).
\end{equation}
\end{lemma}

We will discretize the operator $\tilde{\Lcal}(b(u)p,v) - a_0 \Lcal(q,v)$ in
(\ref{var:nonlinear:1}) explicitly and the other operator $a_0 \Lcal(q,v)$ implicitly.
The fully discrete scheme will be referred to as EIN-LDG scheme in this paper.
In the next subsection we will give a brief introduction of the IMEX RK time discretizations.

\subsection{The  IMEX RK time discretizations}

For a detailed introduction to  IMEX RK schemes, we refer the readers
to \cite{Ascher} and \cite{Cooper}.
To give a brief introduction of the scheme,
let us consider the system of ordinary differential equations
\begin{equation} \label{ode}
\frac{\mathrm{d} \mathbf{y}}{\mathrm{d}t} = L (t, \mathbf{y})
+ N(t, \mathbf{y}), \qquad \mathbf{y}(t_0) = \mathbf{y}_0,
\end{equation}
where $\mathbf{y}=[y_1,y_2,\cdots,y_d]^{\top}$,
$L(t, \mathbf{y})$ and $N(t,\mathbf{y})$ are
derived from the spatial discretization of the two parts of the right hand side of PDEs.
By applying the general $s$-stage IMEX RK time marching scheme,
the solution of (\ref{ode}) advanced from time $t^n$ to
$t^{n+1}=t_n + \dt$ is given by:
\begin{align*}
\label{IMEX0}
& \mathbf{Y}_1 = \mathbf{y}_n,  \nonumber \\
&\mathbf{Y}_i = \mathbf{y}_n +
\dt \sum_{j=1}^i a_{ij} L (t^j_n, \mathbf{Y}_j)
+\dt \sum_{j=1}^{i-1} \hat{a}_{ij} N(t^j_n, \mathbf{Y}_j),
\quad 2\le i \le s+1,
\nonumber \\
& \mathbf{y}_{n+1} = \mathbf{y}_n +
\dt  \sum_{i=1}^{s+1} b_i L (t^i_n, \mathbf{Y}_i)
+\dt \sum_{i=1}^{s+1} \hat{b}_{i} N(t^i_n, \mathbf{Y}_i),
\end{align*}
where $\dt$ is the time step, $\mathbf{Y}_i$ denotes the intermediate stages,
$c_i=\sum_{j=1}^i a_{ij}=\sum_{j=1}^{i-1}\hat{a}_{ij}$,
and $t^j_n = t_n+c_j \dt$.
Denote $A=(a_{ij}), \hat{A}=(\hat{a}_{ij}) \in
\mathbb{R}^{(s+1)\times(s+1)}$,
$\bm{b}^{\top}=[b_1,\cdots,b_{s+1}],
\hat{\bm{b}}^{\top}=[\hat{b}_1,\cdots,\hat{b}_{s+1}]$ and
$\bm{c}^{\top}=[0,c_2,\cdots,c_{s+1}]$,
then we can express the
general $s$-stage IMEX RK scheme as the following Butcher tableau
\begin{equation}\label{IMEX}
\begin{array}{c|c|c}
\bm{c} & A & \hat{A} \\
\hline
       & \bm{b}^{\top} & \hat{\bm{b}}^{\top}
\end{array}
\end{equation}

In the above tableau, the pair $(A \,|\, \bm{b})$ determines an $s$-stage
diagonally implicit RK method and $(\hat{A}\,|\, \hat{\bm{b}})$ defines an
$(s+1)$-stage ($s$-stage if $\hat{b}_{s+1}=0$) explicit RK method.
The first order IMEX RK method is taking the forward Euler
discretization for the explicit part and the backward Euler discretization
for the implicit part, which is expressed in the Butcher tableau
\begin{equation} \label{imex:1}
\begin{array}{c|cc|cc}
0 & 0 & 0 & 0 & 0 \\
1 & 0 & 1 & 1 & 0 \\
\hline
  & 0 & 1 & 1 & 0
\end{array}
\end{equation}
The second order IMEX RK method presented in this paper is
\begin{equation}  \label{imex:2}
{ \renewcommand\arraystretch{1.2}
\begin{array}{c|ccc|ccc}
0 & 0 & 0 & 0 &0  & 0 & 0\\
\tfrac12 & 0 & \tfrac12 & 0 & \tfrac12 & 0 & 0 \\
1 & \tfrac12  & 0 & \tfrac12 & 0 &1 & 0 \\
\hline
  & \tfrac12  & 0 & \tfrac12 & 0 &1 & 0
\end{array} }
\end{equation}
which is a modification of the second order scheme
\begin{equation}\label{cooper}
{\renewcommand\arraystretch{1.2}
\begin{array}{c|ccc|ccc}
0 & 0 & 0 & 0 &0  & 0 & 0\\
\tfrac{\mu}{2} & \tfrac{\mu}{2} & 0 & 0 & \tfrac{\mu}{2} & 0 & 0 \\
1 & \tfrac12  & 0 & \tfrac12 & \tfrac{\mu-1}{\mu} &\tfrac{1}{\mu} & 0 \\
\hline
  & \tfrac12  & 0 & \tfrac12 & \tfrac{\mu-1}{\mu} &\tfrac{1}{\mu} & 0
\end{array} }
\end{equation}
given by \cite{Cooper}, where $\mu \neq 0$.
Notice that if we let $\mu=1$, then (\ref{imex:2}) and (\ref{cooper})
are only different in the discretization of $L (t, \mathbf{y})$
at the first intermediate stage, scheme (\ref{cooper})
discretizes $L (t, \mathbf{y})$ explicitly at the first stage,
while the modified scheme (\ref{imex:2}) discretize
$L (t, \mathbf{y})$ implicitly at the first stage.
%\textcolor[rgb]{0,0,1}
{Owing to the implicit discretization
at the first stage, the stability of the modified scheme (\ref{imex:2})
is better than the original one (\ref{cooper}), especially when
adopting it for the
convection-diffusion problems.  This is why we consider the modified
scheme (\ref{imex:2}) in this paper.}

\section{Stability analysis}
\label{sec3}
\setcounter{equation}{0}
In this section, we will present the stability analysis
for the proposed EIN-LDG schemes.
We would like to investigate how to choose $a_0$ such that the schemes are stable.
For simplicity of analysis, we consider the simplified equation
\begin{equation}
U_t = a U_{xx},
\end{equation}
with constant diffusion coefficient $a>0$.
Adding and subtracting a term $a_0 U_{xx}$ we get
\begin{equation} \label{eq:linear}
U_t  =(a-a_0)U_{xx} + a_0 U_{xx}.
\end{equation}
Then the LDG scheme reads
\bes \label{LDG:simplify}
\begin{align}
(u_t,v) =&\,(a-a_0) \Lcal(q,v)+ a_0 \Lcal(q,v), \\
(q,r) =&\, \Kcal(u,r),
\end{align}
\ees
where $\Lcal$ and $\Kcal$ has been defined in Section \ref{sec2}.

\subsection{First order scheme}

Now we consider the first order EIN-LDG scheme, which is the
first order IMEX time discretization (\ref{imex:1})
coupled with (\ref{LDG:simplify}), i.e,
\begin{subequations}\label{first}
\begin{align}
&(u^{n+1},v)=(u^n,v)+(a-a_0) \dt \Lcal(q^n,v)+a_0 \dt  \Lcal(q^{n+1},v), \label{euler:1}\\
& (q^{\nl},r) = \Kcal(u^{\nl},r), \quad \mbox{for} \quad \ell=0,1,
\end{align}
\end{subequations}
where $w^{n,0}=w^n$ and $w^{n,1}=w^{n+1}$ for $w=u,q$.

%\textcolor[rgb]{0,0,1}
{For the simplified linear model, if we let $a_0=a$ then
the scheme (\ref{first}) degenerates to  backward Euler scheme,
which is unconditionally stable in the sense that
\begin{equation} \label{L2}
\norm{u^n} \le \norm{u^0}, \quad \forall n.
\end{equation}
So we only consider the case $a_0 \neq a$.}
We state the stability result in the following theorem.

\begin{theorem}
If $a_0 \ge \frac{a}{2}$ and $a_0 \neq a$, then the first order EIN-LDG scheme (\ref{first})
is unconditionally stable in the sense that
\begin{equation} \label{main:1}
\norm{u^{n}}^2+ a_0\dt \norm{q^{n}}^2  \le \norm{u^0}^2 + a_0 \dt \norm{q^0}^2.
\end{equation}
\end{theorem}
\begin{proof}
Taking $v=u^{n+1}$ in (\ref{euler:1}), and by the property (\ref{property}) we have
\begin{equation} \label{linear:energy:1}
\frac12 \norm{u^{n+1}}^2 + \frac12 \norm{u^{n+1}-u^n}^2 -\frac12 \norm{u^n}^2
 =-(a-a_0)\dt (q^n,q^{n+1}) -a_0 \dt \norm{q^{n+1}}^2.
\end{equation}
Rearranging the terms yields
\begin{equation*}
LHS=\frac12 \norm{u^{n+1}}^2 + \frac12 \norm{u^{n+1}-u^n}^2 -\frac12 \norm{u^n}^2
+a_0 \dt \norm{q^{n+1}}^2=(a_0-a)\dt (q^n,q^{n+1})=RHS.
\end{equation*}
By  simple use of the Cauchy-Schwarz and the Young's inequalities we get
\begin{equation*}
RHS \le |a_0-a| \dt \norm{q^n} \norm{q^{n+1}} \le \frac{a_0}{2} \dt \norm{q^{n+1}}^2
+ \frac{(a_0-a)^2}{2 a_0} \dt \norm{q^{n}}^2.
\end{equation*}
Hence, if we let $\frac{(a_0-a)^2}{2 a_0} \le \frac{a_0}{2}$, i.e, $a_0 \ge \frac{a}{2}$,
then
\begin{equation*}
LHS \le \frac{a_0}{2}\dt (\norm{q^n}^2+\norm{q^{n+1}}^2).
\end{equation*}
As a result, we have
\begin{equation*}
\frac12 \norm{u^{n+1}}^2 + \frac12 \norm{u^{n+1}-u^n}^2 -\frac12 \norm{u^n}^2
+\frac{a_0}{2} \dt (\norm{q^{n+1}}^2 -\norm{q^n}^2 ) \le 0,
\end{equation*}
that is
\begin{equation*}
\norm{u^{n+1}}^2+ a_0\dt \norm{q^{n+1}}^2  \le \norm{u^n}^2 + a_0 \dt \norm{q^n}^2.
\end{equation*}
And hence we are led to (\ref{main:1}).
\end{proof}

\subsection{Second order scheme}
\label{sec3:2}
The second order EIN-LDG scheme, which is the second order
IMEX scheme (\ref{imex:2})
coupled with the LDG method (\ref{LDG:simplify}), reads
\begin{subequations}\label{second:cooper}
\begin{align}
(u^{n,1},v)=&\,(u^n,v)+ \frac12 (a-a_0) \dt\Lcal(q^n,v) + \frac12 a_0 \dt \Lcal(q^{n,1},v),\label{cooper:1}\\
(u^{n+1},v)=&\,(u^n,v) +(a-a_0)\dt  \Lcal(q^{n,1},v)
+\frac12 a_0 \dt[\Lcal(q^{n},v)+\Lcal(q^{n+1},v)], \label{cooper:2}\\
(q^{\nl},r)=&\,\Kcal(u^{\nl},r), \quad \mbox{for} \quad \ell=0,1,2,
\end{align}
\end{subequations}
where $w^{n,0}=w^n$ and $w^{n,2}=w^{n+1}$ for $w=u,q$.

%\textcolor[rgb]{0,0,1}
{The same as in the first order scheme,
we only consider the case $a_0\neq a$, since in the
case $a_0 =a$ we can also easily get
(\ref{L2}) unconditionally.}
The stability result is given in the following theorem.
\begin{theorem}
If
%\textcolor[rgb]{0,0,1}
{$a_0 \ge \frac{a}{2}$} and $a_0\neq a$,
then the second order EIN-LDG scheme (\ref{second:cooper})
satisfies
\begin{equation} \label{main:2}
\norm{u^{n}}^2+ \frac12 a_0\dt \norm{q^{n}}^2  \le \norm{u^0}^2 + \frac12 a_0 \dt \norm{q^0}^2.
\end{equation}
\end{theorem}
\begin{proof}
Subtracting (\ref{cooper:1}) from (\ref{cooper:2}) we get
\begin{equation} \label{cooper:3}
(u^{n+1}-u^{n,1},v)=(a_0-\frac12 a)\dt \Lcal(q^n,v) + (a-\frac32 a_0)\dt \Lcal(q^{n,1},v) + \frac{a_0}{2}\dt
\Lcal(q^{n+1},v).
\end{equation}

Taking $v=u^{n,1}$ in (\ref{cooper:1}) we have
\begin{equation} \label{1}
\frac12 \norm{u^{n,1}}^2 + \frac12 \norm{u^{n,1}-u^n}^2 - \frac12 \norm{u^n}^2 + \frac12 (a-a_0)  \dt(q^n,q^{n,1})
+\frac12 a_0 \dt \norm{q^{n,1}}^2 = 0,
\end{equation}
where we have used property (\ref{property}).
Taking $v= u^{n+1}$ in (\ref{cooper:3}) we have
\begin{align}\label{2}
&\frac12 \norm{u^{n+1}}^2 + \frac12 \norm{u^{n+1}-u^{n,1}}^2 -
\frac12 \norm{u^{n,1}}^2 +  (a_0-\frac{a}{2})  \dt(q^n,q^{n+1})
\nonumber \\
&+(a-\frac32 a_0)\dt (q^{n,1},q^{n+1})+\frac12 a_0 \dt \norm{q^{n+1}}^2 = 0.
\end{align}
Adding (\ref{1}) and (\ref{2}) together, and multiplying by 2, we get
\begin{align*}
&\norm{u^{n+1}}^2 + \norm{u^{n,1}-u^n}^2 + \norm{u^{n+1}-u^{n,1}}^2 -
\norm{u^{n}}^2 +a_0\dt \left[ \norm{q^{n,1}}^2 + \norm{q^{n+1}}^2\right]\\
& + \dt \left[(a-a_0)(q^n,q^{n,1})+(2a_0-a) (q^n,q^{n+1})
+(2a-3 a_0) (q^{n,1},q^{n+1})\right]=0.
\end{align*}
Then by adding and subtracting $\delta \dt \norm{q^n}^2$ we obtain
\begin{align} \label{3}
&\norm{u^{n+1}}^2 + \norm{u^{n,1}-u^n}^2 + \norm{u^{n+1}-u^{n,1}}^2 - \norm{u^{n}}^2 \nonumber \\
+& \delta \dt (\norm{q^{n+1}}^2-\norm{q^n}^2) + \dt \int_{\Omega} \mathbf{q}^{\top} \mathbb{A} \mathbf{q} \,
\mathrm{d}x=0,
\end{align}
 where  $\mathbf{q}=(q^n,q^{n,1},q^{n+1})^\top$, and
\begin{equation} \label{A}
\mathbb{A}= \begin{pmatrix} \delta & \frac12 (a-a_0) &  a_0-\frac{a}{2}\\
                            \frac12 (a-a_0) & a_0 &  a-\frac32 a_0 \\
                           a_0-\frac{a}{2} & a-\frac32 a_0  &  a_0 -\delta
                           \end{pmatrix}.
\end{equation}

On the other hand, taking $v=u^{n,1}-u^n$ in (\ref{cooper:1}) we have
\begin{equation*}
\norm{u^{n,1}-u^n}^2 + \frac12 (a-a_0)\dt (q^n,q^{n,1}-q^n) + \frac12 a_0\dt (q^{n,1},q^{n,1}-q^n) = 0,
\end{equation*}
owing to (\ref{property}). That is
\begin{equation} \label{4}
\norm{u^{n,1}-u^n}^2  + \dt \int_{\Omega} \mathbf{q}^{\top} \mathbb{B} \mathbf{q} \,
\mathrm{d}x=0,
\end{equation}
where
\begin{equation} \label{B}
\mathbb{B}= \begin{pmatrix}  \frac12 (a_0-a) &  \frac{a}{4}-\frac{a_0}{2} & 0\\
                            \frac{a}{4}-\frac{a_0}{2} & \frac12 a_0 &  0 \\
                           0 & 0   & 0
                           \end{pmatrix}.
\end{equation}

Adding (\ref{3}) and $\sigma \times$(\ref{4}) together  leads to
\begin{align}
&\,\norm{u^{n+1}}^2 + (1+\sigma)\norm{u^{n,1}-u^n}^2 +
\norm{u^{n+1}-u^{n,1}}^2 - \norm{u^{n}}^2 \nonumber \\
&\,+ \delta \dt (\norm{q^{n+1}}^2-\norm{q^n}^2) + \dt \int_{\Omega} \mathbf{q}^{\top}
 (\mathbb{A}+\sigma \mathbb{B}) \mathbf{q} \,
\mathrm{d}x=0.
\end{align}
Here $0 \le \delta \le a_0$ and $\sigma > -1$ are free parameters.
For convenience, we let $\delta = \frac12 a_0$.
We claim that there exists $\sigma >-1$ such that
the matrix $\mathbb{A}+\sigma \mathbb{B}$ is positive definite
for any $a_0 > \frac12 a$, whose proof will be deferred to
Lemma \ref{lemma:1}.
%\textcolor[rgb]{0,0,1}
{We can also verify that
$\mathbb{A}+\sigma \mathbb{B}$ is semi-positive definite
for $a_0=\frac12 a$ if $\sigma=0$, since the eigenvalues of
the matrix are $\frac34 a, \frac14 a$ and $0$ in this situation.}
Thus we can get
\begin{equation}
\norm{u^{n+1}}^2 + \frac12 a_0 \dt \norm{q^{n+1}}^2  \le \norm{u^{n}}^2 + \frac12 a_0 \dt \norm{q^{n}}^2.
\end{equation}
And hence we obtain (\ref{main:2}).
\end{proof}

\begin{lemma} \label{lemma:1}
Let $\delta=\frac12 a_0$, for any $a_0 > \frac12 a$ and $a_0 \neq a$, there exists $\sigma >-1$ such that
$\mathbb{A}+\sigma \mathbb{B}$ is positive definite, where $\mathbb{A}$ and $\mathbb{B}$
are defined in (\ref{A}) and (\ref{B}) respectively.
\end{lemma}
\begin{proof}
Assume $a_0=\theta a$, then
 %If $\theta=1$ then it is easy to check that
%$\mathbb{A}+\sigma \mathbb{B}$ is positive definite
%by choosing $\sigma = 0$. In what follows, we assume
%and assume $\theta>\frac12$ and $\theta\neq 1$. Then
\begin{equation*}
\mathbb{A}+\sigma \mathbb{B}=\frac12 a \begin{pmatrix} \theta+\sigma (\theta-1) &  1-\theta+\sigma(\frac12-\theta) & 2\theta-1\\
                            1-\theta+\sigma(\frac12-\theta)  & (2+\sigma)\theta &  2-3\theta \\
                          2\theta-1 & 2-3\theta  & \theta
                           \end{pmatrix}.
\end{equation*}
To ensure $\mathbb{A}+\sigma \mathbb{B}$ is positive definite, we require all the leading principle minors
 are positive, namely
\begin{subequations} \label{8}
 \begin{align}& \sigma (\theta-1)+\theta>0,  \label{5} \\
              &  -\frac14 \sigma^2 +(\theta^2+\theta-1)\sigma+\theta^2+2\theta-1>0, \label{6} \\
              &  -\frac14 \theta \sigma^2 +(6\theta^2-7\theta+2)\sigma -(4\theta^3+4\theta^2-11\theta+4)>0. \label{7}
  \end{align}
\end{subequations}
In what follows, we will prove the solution ($\sigma$) of (\ref{8}) exists provided that $\theta>\frac12$ and $\theta\neq 1$.
From (\ref{6}), we get
\begin{equation} \label{9}
2(\theta^2+\theta-1)-2\sqrt{\Delta_1}<\sigma<2(\theta^2+\theta-1)+2\sqrt{\Delta_1},
\end{equation}
where $\Delta_1 = (\theta^2+\theta-1)^2+\theta^2+2\theta-1=\theta^4+2\theta^3$, which is always positive if $\theta>\frac12$. From (\ref{7}) we get
\begin{equation} \label{10}
\frac{2(6\theta^2-7\theta+2)-2\sqrt{\Delta_2}}{\theta}<\sigma<\frac{2(6\theta^2-7\theta+2)+2\sqrt{\Delta_2}}{\theta},
\end{equation}
where $\Delta_2 = (6\theta^2-7\theta+2)^2-\theta(4\theta^3+4\theta^2-11\theta+4)=32\theta^4 -88 \theta^3 + 84 \theta^2 - 32 \theta +4$
which is positive if $\theta>\frac12$ and $\theta\neq 1$.

To simplify the notations,
we denote $A_1=2(\theta^2+\theta-1)-2\sqrt{\Delta_1}$, $A_2=2(\theta^2+\theta-1)+2\sqrt{\Delta_1}$,
$B_1=\frac{2(6\theta^2-7\theta+2)-2\sqrt{\Delta_2}}{\theta}$, and
$B_2=\frac{2(6\theta^2-7\theta+2)+2\sqrt{\Delta_2}}{\theta}$. Since
$$
(A_1-B_2)(A_2-B_1)=-\frac{4}{\theta}\left[ 2\sqrt{\Delta_1 \Delta_2} + (12 \theta^4-5 \theta^3 -24 \theta^2 + 28 \theta -8)\right]
<0 \quad \mbox{if} \quad \theta> \frac12,
$$
we can conclude that the intersection of (\ref{9}) and (\ref{10}) is not empty.

In addition, from (\ref{5}) we get
\begin{equation}
\begin{cases} \sigma < \frac{\theta}{1-\theta} & \mbox{if} \quad \frac12 < \theta<1, \\
              \sigma > - \frac{\theta}{\theta-1} & \mbox{if} \quad \theta>1.
              \end{cases}
\end{equation}
So, if $\frac12 <\theta<1$, then we require $\max\{A_1,B_1\}< \frac{\theta}{1-\theta}$.
This condition can be verified by noticing that
$$
2(\theta^2+\theta-1)-\frac{\theta}{1-\theta}<0 \quad \mbox{and}
\quad \frac{2(6\theta^2-7\theta+2)}{\theta}-\frac{\theta}{1-\theta} < 0,
$$
for $\frac12 <\theta<1$.
If $\theta>1$, we require $\min\{A_2,B_2\}> -\frac{\theta}{\theta-1}$, it holds obviously
since
$$
2(\theta^2+\theta-1)+\frac{\theta}{\theta-1}>0 \quad \mbox{and}
\quad \frac{2(6\theta^2-7\theta+2)}{\theta}+\frac{\theta}{\theta-1} > 0,
$$
for $\theta>1$. Thus we proved that the solution of (\ref{8}) exists.

Furthermore, we can check that
$\min\{A_2,B_2\} >-1 $ in the case $\theta>1$, and in the case $\frac12<\theta<1$,
$\min \{A_2,B_2, \frac{\theta}{1-\theta}\} >-1 $, so we complete the proof of this lemma.
\end{proof}

%\textcolor[rgb]{0,0,1}
{\begin{Remark}
In the above stability analysis for the linear model, it is required to study
the positive definiteness of the matrix $\mathbb{A}+\sigma \mathbb{B}$ which
is a constant matrix. The arguments, however, are not easy to extend to
nonlinear problems, since the corresponding matrix will depend on the numerical
solutions at different intermediate time stages, it will be more complicated
to study the positive definiteness of the matrix. So we need to seek
new techniques to overcome the difficulties, which will be left for
future work.
Even though the analysis for the nonlinear model is not available at present,
the stability analysis for the linear model can provide us with some guidance in designing
schemes for nonlinear diffusion problems.
\end{Remark}}

\section{Optimal error estimates}
\label{sec4}
\setcounter{equation}{0}

With the stability result in the previous section, it is conceptually
straightforward to obtain error estimates for smooth solutions
of the simplified model (\ref{eq:linear}) with $a>0$ being a constant.
We will only give the error estimates for the
second order EIN-LDG scheme (\ref{second:cooper}) as an example.
To this end, we would like to introduce two Gauss-Radau projections, from
$\Ht=\big\{\phi\in L^2(\Omega): \phi|_{I_j}\in H^1(I_j),\,\forall
j=1,\ldots,N\,\big\}$ to $V_h$, denoted by
$\pi_h^{-}$ and $\pi_h^{+}$ respectively.
For any function $p\in \Ht$,
the projections $\pi_h^{\pm}p$ are defined as the unique
element in $V_h$ such that
\begin{subequations}\label{projection}
\begin{alignat}{1}
&  (\pi_h^{-}p-p,v)_{I_j}=0,  \quad\forall v\in \mathcal{P}_{k-1}(I_j),
\quad (\pi_h^{-}p)_{\jfhalf}^{-}=p_{\jfhalf}^{-},
\\
&   (\pi_h^{+}p-p,v)_{I_j}=0,  \quad\forall v\in \mathcal{P}_{k-1}(I_j),
\quad (\pi_h^{+}p)_{\jbhalf}^{+}=p_{\jbhalf}^{+},
\end{alignat}
\end{subequations}
for any $j=1,2,\cdots,N$.
In view of the exact collocation on one endpoint of each element, the
Gauss-Radau projections provide a great help to obtain the optimal
error estimates.

Denote by $\eta=p-\pi_h^{\pm} p$ the projection error.
By a standard scaling argument \cite{Ciarlet}, it is easy to obtain
the following approximation property
\begin{equation} \label{estimate:eta}
\norm{\eta} \leq C h^{\min(k+1,s)} \norm{p}_{H^s},
\end{equation}
where the bounding constant $C>0$ is independent of $h$. Furthermore,
by the definition of the operators $\Lcal$ and $\Kcal$ we have
\begin{equation} \label{proj:property}
\Lcal(p-\pi_h^+p, v)=\Kcal(p-\pi_h^- p, v)=0
\end{equation}
for any $p\in \Ht$ and $v \in V_h$, due to the periodic boundary condition.

Following \cite{WSZ:1},
we introduce three ``reference'' functions, denoted by
$\bm{W}^{(\ell)}=(U^{(\ell)},Q^{(\ell)})$, $\ell=0,1,2$,
associated with the second order IMEX RK time discretization (\ref{imex:2}).
In detail, $U^{(0)}=U$ is the exact solution of problem \eqref{eq:linear}
and then we define
\begin{subequations}\label{Ref:function}
\begin{align}
U^{(1)}&=\, U^{(0)}+\frac12(a-a_0) \dt Q^{(0)}_x+
\frac12 a_0 \dt Q^{(1)}_x,
\label{ref:1} \\
U^{(2)}&=\, U^{(0)}+(a-a_0) \dt Q^{(1)}_x
+\frac12 a_0 \dt (Q^{(0)}_x+Q^{(2)}_x),
\label{ref:2}
\end{align}
where
\begin{equation}
Q^{(\ell)}= U^{(\ell)}_x, \quad \mbox{for} \quad \ell=0,1,2.
\end{equation}
\end{subequations}
For any indices $n$ and $\ell$ under consideration,
the reference function at each stage time level is defined as
$\bm{W}^{n,\ell}=(U^{\nl},Q^{\nl})=\bm{W}^{(\ell)}(x,t^n)$.
Here $\bm{W}^{n,0}=\bm{W}^n$ and $\bm{W}^{n,2}=\bm{W}^{n+1}$.

At each stage time, we denote the error between the exact
(reference) solution and the numerical solution by
${\bm{e}}^{\nl}=(e_u^{\nl},e_q^{\nl})=(U^{\nl}-u^{\nl},Q^{\nl}-q^{\nl})$.
As the standard treatment in finite element analysis, we would like
to divide the error in the form $\bm{e}=\bm{\xi}-\bm{\eta}$, where
\begin{equation}
\bm{\eta}=(\eta_u,\eta_q)=(\pi_h^{-}U-U,\pi_h^{+}Q-Q),
\quad
\bm{\xi}=(\xi_u,\xi_q)=(\pi_h^{-}U-u,\pi_h^{+}Q-q),
\end{equation}
here we have dropped the superscripts $n$ and $\ell$ for simplicity.

We would like to assume that the exact solution $U$ satisfies the
following smoothness
\begin{equation}\label{smooth:assumption}
U \in L^{\infty}(0,T;H^{k+2}),\quad
D_t U \in L^{\infty}(0,T;H^{k+1}),\quad
\mbox{and}\quad
D_t^3 U \in L^{\infty}(0,T;L^2),
\end{equation}
where $D^{\ell}_t U$ is the $\ell$-th order time derivative of $U$.

By the smoothness assumption (\ref{smooth:assumption}),
it follows from (\ref{estimate:eta}) that
the stage projection errors  satisfy
\begin{subequations} \label{estimate:eta:stage}
\begin{equation}
\norm{\eta_u^{n,\ell}} + \norm{\eta_q^{n,\ell}} \le
C h^{k+1}(\norm{U^{\nl}}_{H^{k+1}}+\norm{Q^{\nl}}_{H^{k+1}})\le Ch^{k+1},
\end{equation}
for any $n$ and $\ell=0,1,2$ under consideration. And
owing to the linear structure of the Gauss-Radau projection, we have
\begin{align}
\norm{ \eta_u^{n,1}-\eta_u^n} \le &\, C h^{k+1}\norm{U^{n,1}-U^n}_{H^{k+1}}
\le C h^{k+1}\dt,
\\
\norm{\eta_u^{n+1}-\eta_u^{n,1}}  \le&\,C h^{k+1}\norm{U^{n+1}-U^{n,1}}_{H^{k+1}}\le  Ch^{k+1}\dt.
\end{align}
\end{subequations}
Here the bounding constant $C>0$ depends solely on the smoothness
of the exact solution and is independent of $n, h, \dt$.
%The second
%inequality in (\ref{estimate:eta:stage}) holds since the Gauss-Radau projection
%is defined independent of time, from (\ref{estimate:eta}) we have
%$\norm{\eta_u^{n,1}-\eta_u^n}\le C h^{k+1}\norm{U^{n,1}-U^n}_{H^{k+1}}\le C h^{k+1}\dt \norm{D_t U^n}_{H^{k+1}}$

In what follows we will focus  on the estimate of the error $\bm{\xi}$.
Notice that the ``reference'' function satisfies the following variational forms
\begin{subequations}\label{second:ref}
\begin{align}
(U^{n,1},v)=&\,(U^n,v)+ \frac12 (a-a_0) \dt\Lcal(Q^n,v) + \frac12 a_0 \dt \Lcal(Q^{n,1},v),\label{sec:ref1}\\
(U^{n+1},v)=&\,(U^n,v) +(a-a_0)\dt  \Lcal(Q^{n,1},v)
+\frac12 a_0 \dt[\Lcal(Q^{n},v)+\Lcal(Q^{n+1},v)]+(\zeta^n,v), \label{sec:ref2}\\
(Q^{\nl},r)=&\,\Kcal(U^{\nl},r), \quad \mbox{for} \quad \ell=0,1,2,
\end{align}
\end{subequations}
where $\zeta^n=\Ocal(\dt^3)$ by the smoothness assumption (\ref{smooth:assumption}).

Subtracting these variational forms from those in scheme (\ref{second:cooper}), in the same order,
we obtain the following error equations
\begin{subequations}\label{error:xi}
\begin{align}
(\xi_u^{n,1},v)=&\,(\xi_u^n,v)+(\eta_u^{n,1}-\eta_u^n,v)+
 \frac12 (a-a_0) \dt\Lcal(\xi_q^n,v) + \frac12 a_0 \dt \Lcal(\xi_q^{n,1},v),\label{xi:1}\\
(\xi_u^{n+1},v)=&\,(\xi_u^n,v)+(\eta_u^{n+1}-\eta_u^n,v) +(a-a_0)\dt  \Lcal(\xi_q^{n,1},v) \nonumber \\
&\,+\frac12 a_0 \dt[\Lcal(\xi_q^{n},v)+\Lcal(\xi_q^{n+1},v)]+(\zeta^n,v), \label{xi:2}\\
(\xi_q^{\nl},r)=&\,(\eta_q^{\nl},r)+\Kcal(\xi_u^{\nl},r),\label{xi:q} \quad \mbox{for} \quad \ell=0,1,2,
\end{align}
\end{subequations}
since  $\Lcal(\eta_q,v)=\Kcal(\eta_u,r)=0$ by  property (\ref{proj:property}).

Subtracting (\ref{xi:1}) from (\ref{xi:2}) we get
\begin{align} \label{xi:3}
(\xi_u^{n+1}-\xi_u^{n,1},v)=&\,(a_0-\frac12 a)\dt \Lcal(\xi_q^n,v)
 + (a-\frac32 a_0)\dt \Lcal(\xi_q^{n,1},v) + \frac{a_0}{2}\dt
\Lcal(\xi_q^{n+1},v) \nonumber \\
&\,+(\eta_u^{n+1}-\eta_u^{n,1},v)+(\zeta^n,v).
\end{align}

Taking $v=\xi_u^{n,1}$ in (\ref{xi:1}), $v=\xi_u^{n+1}$ in (\ref{xi:3}), then proceeding along the
similar line as the stability analysis in Subsection \ref{sec3:2},  we obtain
\begin{align} \label{11}
&\,\norm{\xi_u^{n+1}}^2 + \norm{\xi_u^{n,1}-\xi_u^n}^2 + \norm{\xi_u^{n+1}-\xi_u^{n,1}}^2 - \norm{\xi_u^{n}}^2 \nonumber \\
&\,+ \delta \dt (\norm{\xi_q^{n+1}}^2-\norm{\xi_q^n}^2) + \dt \int_{\Omega} \bm{\xi}_q^{\top} \mathbb{A} \bm{\xi}_q \,
\mathrm{d}x=T_1,
\end{align}
 where  $\bm{\xi}_q=(\xi_q^n,\xi_q^{n,1},\xi_q^{n+1})^\top$,
$\mathbb{A}$ is the same as  in (\ref{A}), and
\begin{align*}
 T_1=&\, 2(\eta_u^{n,1}-\eta_u^n, \xi_u^{n,1})+(a-a_0)\dt (\eta_q^{n,1},\xi_q^n)+a_0 \dt (\eta_q^{n,1},\xi_q^{n,1})
 \nonumber \\
 &\,+2(\eta_u^{n+1}-\eta_u^{n,1},\xi_u^{n+1})+(2a_0-a)\dt (\eta_q^{n+1},\xi_q^n)+(2a-3a_0)\dt(\eta_q^{n+1},\xi_q^{n,1})
 \nonumber \\
 &\, + a_0 \dt (\eta_q^{n+1},\xi_q^{n+1}) + 2(\zeta^n,\xi_u^{n+1}).
 %\nonumber \\
% \le&\, \varepsilon \dt (\norm{\xi_u^{n,1}}^2 + \norm{\xi_u^{n+1}}^2)
% +\varepsilon \dt (\norm{\xi_q^{n}}^2 + \norm{\xi_q^{n,1}}^2+\norm{\xi_q^{n+1}}^2)
% + C(h^{2k+2}\dt+\dt^5).
 \end{align*}

On the other hand, taking $v=\xi_u^{n,1}-\xi_u^n$ in (\ref{xi:1}) we get
\begin{equation}\label{12}
\norm{\xi_u^{n,1}-\xi_u^n}^2 + \dt \int_{\Omega} \bm{\xi}_q^{\top} \mathbb{B} \bm{\xi}_q \,
\mathrm{d}x=T_2,
\end{equation}
where $\mathbb{B}$ is the same as in (\ref{B}) and
\begin{align}
T_2=&\,(\eta_u^{n,1}-\eta_u^n,\xi_u^{n,1}-\xi_u^{n})+\frac12 (a-a_0)\dt(\eta_q^{n,1}-\eta_q^n,\xi_q^n)
+\frac12 a_0 \dt (\eta_q^{n,1}-\eta_q^n,\xi_q^{n,1}).
\end{align}
A simple use of the Cauchy-Schwarz and Young's inequalities and the properties (\ref{estimate:eta:stage}) we have
$$|T_1+T_2|\le \varepsilon \dt (\norm{\xi_u^n}^2+\norm{\xi_u^{n,1}}^2 + \norm{\xi_u^{n+1}}^2)
 +\varepsilon \dt (\norm{\xi_q^{n}}^2 + \norm{\xi_q^{n,1}}^2+\norm{\xi_q^{n+1}}^2)
+ C(h^{2k+2}\dt+\dt^5),$$
for arbitrary $\varepsilon$.
So adding (\ref{11}) and $\sigma \times$(\ref{12}) together  leads to
\begin{align}
&\norm{\xi_u^{n+1}}^2 + (1+\sigma)\norm{\xi_u^{n,1}-\xi_u^n}^2 +
\norm{\xi_u^{n+1}-\xi_u^{n,1}}^2 - \norm{\xi_u^{n}}^2 \nonumber \\
&+ \delta \dt (\norm{\xi_q^{n+1}}^2-\norm{\xi_q^n}^2) +
\dt \int_{\Omega} \bm{\xi}_q^{\top} (\mathbb{A}+\sigma \mathbb{B}-\varepsilon \mathbb{I}) \bm{\xi}_q \,
\mathrm{d}x \nonumber \\
 \le&\, \varepsilon \dt (\norm{\xi_u^n}^2+\norm{\xi_u^{n,1}}^2 + \norm{\xi_u^{n+1}}^2)
+ C(h^{2k+2}\dt+\dt^5)
\nonumber \\
\le&\, \varepsilon \dt (\norm{\xi_u^n}^2+\norm{\xi_u^{n,1}-\xi_u^n}^2 + \norm{\xi_u^{n+1}-\xi_u^{n,1}}^2)
+ C(h^{2k+2}\dt+\dt^5).
\end{align}
Here $0 \le \delta \le a_0$ and $\sigma > -1$ are free parameters, $\mathbb{I}$ is the identity matrix.

As in the stability analysis, we let $\delta = \frac12 a_0$.
%\textcolor[rgb]{0,0,1}
{Since the matrix $\mathbb{A}+\sigma \mathbb{B}$ is symmetric,
from Lemma \ref{lemma:1}
%By the matrix theory \cite{book} and Lemma \ref{lemma:1}
we conclude that, for $a_0 > \frac12 a$ there exists $\sigma >-1$ such that
the matrix $\mathbb{A}+\sigma \mathbb{B}-\varepsilon \mathbb{I}$ is also positive definite,
by choosing $\varepsilon$ small enough such that $\varepsilon \le \frac12 \lambda$,
where $\lambda$ is the smallest positive eigenvalue of the
matrix $\mathbb{A}+\sigma \mathbb{B}$.
Note that in the case $a_0 =\frac12 a$, we are not able to ensure the positive definiteness of
the matrix $\mathbb{A}+\sigma \mathbb{B}-\varepsilon \mathbb{I}$, since in this case
the matrix  $\mathbb{A}+\sigma \mathbb{B}$ is only semi-positive definite.}
Thus for $a_0 > \frac12 a$, using the discrete Gronwall's inequality yields
\begin{equation}
\norm{\xi_u^{n+1}}^2 + \frac12 a_0 \dt \norm{\xi_q^{n+1}}^2  \le
\norm{\xi_u^0}^2+\frac12 a_0 \dt \norm{\xi_q^{0}}^2  + C(h^{2k+2}+\dt^4).
\end{equation}
Taking $u^0=\pi_h^- U^0$ we get $\xi_u^0=0$ and hence from (\ref{xi:q})
we get $\norm{\xi_q^0}\le \norm{\eta_q^0}\le Ch^{k+1}$,
so we are led to
\begin{equation} \label{est:xi}
\norm{\xi_u^n} \le C(h^{k+1}+\dt^2).
\end{equation}

Finally we obtain the following theorem by (\ref{estimate:eta:stage}), (\ref{est:xi})
and the triangle inequality.

\begin{theorem}
Let $U(x,t)$ be the exact solution of equation (\ref{eq:linear}) satisfying the
smoothness assumption (\ref{smooth:assumption}), and let $u^n$ be the solution
of the second order EIN-LDG scheme (\ref{second:cooper}). Then
if $a_0 > \frac{a}{2}$ we have
\begin{equation} \label{main:3}
\max_{n \dt \le T}\norm{U(x,t^n)-u^{n}}  \le C(h^{k+1}+\dt^2),
\end{equation}
where $C$ is a bounding constant independent of $n,h,\dt$.
\end{theorem}

\section{Numerical Experiments}
\label{sec5}
\setcounter{equation}{0}

In this section, we will numerically validate the accuracy and performance of
the LDG spatial discretization (\ref{var:nonlinear}) coupled with
the first and second order IMEX schemes (\ref{imex:1}) and (\ref{imex:2}).
In addition, we would like to test for a third order IMEX scheme proposed in \cite{Ascher}
\begin{equation} \label{imex:3}
{\renewcommand\arraystretch{1.3}
\begin{array}{c|ccccc|ccccc}
0 &0  & 0 & 0 & 0 &0  & 0 & 0 & 0 &0 & 0\\
\frac12 &0& \frac12 & 0 & 0 & 0 & \frac12 & 0 & 0 & 0 & 0 \\
\frac23 & 0 & \frac16 & \frac12 & 0 & 0 &
\frac{11}{18} &\frac{1}{18} & 0 & 0 & 0 \\
\frac12 & 0 & -\frac12 & \frac12 & \frac12 & 0 & \frac56 & -\frac56 & \frac12 & 0& 0 \\
1 & 0 & \frac32 &-\frac32 & \frac12 & \frac12 & \frac14 & \frac74 & \frac34 & -\frac74 & 0 \\
\hline
  & 0 & \frac32 &-\frac32 & \frac12 & \frac12 & \frac14 & \frac74 & \frac34 & -\frac74 & 0
\end{array} }
\end{equation}

\subsection{The stability and accuracy test}

In this subsection we test the stability and accuracy of the proposed schemes.
We will consider two examples. In each example, the source term  $f(x,t)$
is chosen properly such that the exact solution satisfies the given equation.
The final computing time is $T=10$ and uniform meshes are adopted
for all tests in this subsection. In addition, we take piecewise
constant, piecewise linear and piecewise quadratic polynomials in
the LDG spatial discretization for the first
order, the second order and the third order IMEX time discretization,
respectively,
%\textcolor[rgb]{0,0,1}
{such that the orders accuracy of errors in
space and time match if the time step $\dt=\Ocal(h)$.}

\medskip

\noindent{\emph{Example 1}}. The diffusion equation $u_t =(a(u) u_x)_x+f(x,t)$
with exact solution
$$u(x,t)=\sin(x-t)$$
defined on $[-\pi,\pi]$. We will consider
three cases:

\noindent (i) $a(u)=\frac12$, \quad (ii) $a(u)=u^2+1$, \quad (iii) $a(u)=\sin^2 u$.

For this example, the time step is $\dt=h$,
where $h=2\pi/N$ is the mesh size, $N$ is the number of elements.

In Tables \ref{table1}-\ref{table3}, we list the $L^2$ errors and
orders of accuracy for the three cases.
 In each table, we display the numerical results of the
three IMEX schemes (\ref{imex:1}), (\ref{imex:2}) and (\ref{imex:3})
coupled with the LDG method (\ref{var:nonlinear}) with different $a_0$.
From these tables, we see that the first and second order EIN-LDG schemes
are stable and can achieve optimal error accuracy
%\textcolor[rgb]{0,0,1}
{in both space and time} if
$a_0 \ge \max\{a(u^0)\}/2$,
where $u^0$ is the approximation of the initial solution.
From the experiment we also find that the smallest $a_0$ to ensure the
stability of the third order EIN-LDG scheme is about $0.54 \max\{ a(u^0)\}$,
and we observe optimal error accuracy in both space and time if $a_0 \ge 0.54\max\{a(u^0)\}$.
From the numerical results we can also find that larger $a_0$ may cause
larger errors.
%The results also imply that lager $a_0$ may cause more numerical dissipation.

\begin{table}[htp]
\caption{The $L^2$ errors and orders of accuracy for Example 1: $a(u)=1/2$.}
\begin{center}\footnotesize
\renewcommand{\arraystretch}{1.2}
\begin{tabular}{c|c|cc|cc|cc|cc}
\hline
\multirow{2}{*}{schemes}&\multirow{2}{*}{$N$}
& $L^2$ error &  order    & $L^2$ error &  order
& $L^2$ error & order    & $L^2$ error & order \\
\cline{3-10}
 &          & \multicolumn{2}{|c|}{$a_0=0.24$}
&\multicolumn{2}{c|}{$a_0=0.25$} &\multicolumn{2}{c|}{$a_0=1$}
&\multicolumn{2}{c}{$a_0=10$}\\
\hline
%&40   &1.61E-01 & -     &1.62E-01 & -    &2.68E-01 & -    &1.20E+00 &-  \\
&80   &8.04E-02&    -  & 8.09E-02&    -    &1.40E-01&    -& 7.94E-01&    -\\
(\ref{imex:1})&160  &4.02E-02&    1.00& 4.04E-02&    1.00    &7.13E-02&    0.97& 4.84E-01&    0.71\\
with &320  &2.01E-02&    1.00& 2.02E-02&    1.00   &3.60E-02&    0.98& 2.73E-01&    0.83\\
$k=0$&640  & 1.73E+11&  -42.97& 1.01E-02&    1.00   &1.81E-02&    0.99& 1.46E-01&    0.90\\
&1280  &4.33E+46& -117.59& 5.05E-03&    1.00    &9.08E-03&    1.00& 7.56E-02&    0.95\\
\hline
  \multicolumn{2}{c|}{}    & \multicolumn{2}{|c|}{$a_0=0.24$}
&\multicolumn{2}{c|}{$a_0=0.25$} &\multicolumn{2}{c|}{$a_0=1$}
&\multicolumn{2}{c}{$a_0=10$}\\
\hline
%&40   &8.72E+01 & -     &3.52E-03 & -    &7.65E-03 & -    &3.11E-01 &-  \\
&80   &3.90E+09&  - & 8.80E-04&    -    &2.02E-03&    -& 1.24E-01&    -\\
(\ref{imex:2}) &160  &3.09E+25&  -52.82& 2.19E-04&    2.00    &5.17E-04&    1.97& 4.13E-02&    1.58\\
with&320  & 7.67E+57& -107.61& 5.48E-05&    2.00   &1.31E-04&    1.98& 1.22E-02&    1.76\\
$k=1$&640  & 1.53+123& -216.92& 1.37E-05&    2.00   &3.30E-05&    1.99& 3.33E-03&    1.87\\
&1280  &Infinity&    -Inf& 3.42E-06&    2.00    &8.28E-06&    1.99& 8.73E-04&    1.93\\
\hline
 \multicolumn{2}{c|}{}        & \multicolumn{2}{|c|}{$a_0=0.26$}
&\multicolumn{2}{c|}{$a_0=0.27$} &\multicolumn{2}{c|}{$a_0=1$}
&\multicolumn{2}{c}{$a_0=10$}\\
\hline
&80   &6.37E+12 & -     &6.92E-06 & -    &5.10E-05 & -    &3.31E-02 &-  \\
(\ref{imex:3})&160   &3.67E+31&  -62.32 & 8.67E-07&    3.00    &6.58E-06&    2.95& 6.49E-03&    2.35\\
with &320  &1.26E+70& -128.01& 1.09E-07&    2.99    &8.41E-07&    2.97& 1.05E-03&    2.63\\
$k=2$&640  & 8.64+147& -258.56& 1.36E-08&    3.00   &1.06E-07&    2.99& 1.51E-04&    2.80\\
&1280  & NaN&     NaN& 1.73E-09&    2.98   &1.31E-08&    3.02& 2.03E-05&    2.90\\
\hline
\end{tabular}
\end{center}
\label{table1}
\end{table}

\begin{table}[htp]
\caption{The $L^2$ errors and orders of accuracy for Example 1: $a(u)=u^2+1$.}
\begin{center}\footnotesize
\renewcommand{\arraystretch}{1.2}
\begin{tabular}{c|c|cc|cc|cc|cc}
\hline
\multirow{2}{*}{schemes}&\multirow{2}{*}{$N$}
& $L^2$ error &  order    & $L^2$ error &  order
& $L^2$ error & order    & $L^2$ error & order \\
\cline{3-10}
 &          & \multicolumn{2}{|c|}{$a_0=0.85$}
&\multicolumn{2}{c|}{$a_0=0.9$} &\multicolumn{2}{c|}{$a_0=1$}
&\multicolumn{2}{c}{$a_0=10$}\\
\hline
%&40   &1.61E-01 & -     &1.62E-01 & -    &2.68E-01 & -    &1.20E+00 &-  \\
&80   &9.73E-02&    -  & 1.01E-01&    -    &1.08E-01&    -& 6.72E-01&    -\\
(\ref{imex:1})&160  &4.90E-02&    0.99& 5.09E-02&    0.99    &5.47E-02&    0.98& 3.91E-01&    0.78\\
with &320  &2.46E-02&    0.99&  2.56E-02&    0.99   &2.75E-02&    0.99& 2.11E-01&    0.89\\
$k=0$&640  & NaN&     NaN& 1.28E-02&    1.00   &1.38E-02&    1.00& 1.10E-01&    0.94\\
&1280  &NaN&     NaN& 6.42E-03&    1.00    &6.90E-03&    1.00& 5.61E-02&    0.97\\
\hline
  \multicolumn{2}{c|}{}    & \multicolumn{2}{|c|}{$a_0=0.95$}
&\multicolumn{2}{c|}{$a_0=0.98$} &\multicolumn{2}{c|}{$a_0=1$}
&\multicolumn{2}{c}{$a_0=10$}\\
\hline
%&40   &8.72E+01 & -     &3.52E-03 & -    &7.65E-03 & -    &3.11E-01 &-  \\
&80    & 1.04E-03&    -   &1.03E-03&    - &1.02E-03&    -& 8.64E-02&    -\\
(\ref{imex:2}) &160  & 2.62E-04&    1.99 & 2.59E-04&    1.99   &2.57E-04&    1.98& 2.83E-02&    1.61\\
with&320  & 6.65E-05&    1.98 & 6.61E-05&    1.97  &6.56E-05&    1.97& 8.32E-03&    1.76\\
$k=1$&640  & NaN&     NaN & 1.65E-05&    2.00 &1.64E-05&    2.00& 2.27E-03&    1.88\\
&1280  & NaN&     NaN  & 4.12E-06&    2.00  &4.13E-06&    1.99& 5.95E-04&    1.93\\
\hline
 \multicolumn{2}{c|}{}        & \multicolumn{2}{|c|}{$a_0=1$}
&\multicolumn{2}{c|}{$a_0=1.05$} &\multicolumn{2}{c|}{$a_0=1.1$}
&\multicolumn{2}{c}{$a_0=10$}\\
\hline
&80   &3.56E-05 & -     &3.88E-05 & -    &4.23E-05 & -    &2.17E-02 &-  \\
(\ref{imex:3})&160   &4.73E-06&    2.91 & 5.23E-06&    2.89   &5.77E-06&    2.87& 4.25E-03&    2.35\\
with &320  &NaN&     NaN& 7.45E-07&    2.81    &8.30E-07&    2.80& 6.97E-04&    2.61\\
$k=2$&640  & NaN&     NaN& 9.52E-08&    2.97   &1.07E-07&    2.96& 1.01E-04&    2.79\\
&1280  & NaN&     NaN& 1.27E-08&    2.91   &1.40E-08&    2.93& 1.37E-05&    2.88\\
\hline
\end{tabular}
\end{center}
\label{table2}
\end{table}

\begin{table}[htp]
\caption{The $L^2$ errors and orders of accuracy for Example 1: $a(u)=\sin^2u$.}
\begin{center}\footnotesize
\renewcommand{\arraystretch}{1.2}
\begin{tabular}{c|c|cc|cc|cc|cc}
\hline
\multirow{2}{*}{schemes}&\multirow{2}{*}{$N$}
& $L^2$ error &  order    & $L^2$ error &  order
& $L^2$ error & order    & $L^2$ error & order \\
\cline{3-10}
 &          & \multicolumn{2}{|c|}{$a_0=0.5\sin^2 1-0.1$}
&\multicolumn{2}{c|}{$a_0=0.5\sin^2 1$} &\multicolumn{2}{c|}{$a_0=1$}
&\multicolumn{2}{c}{$a_0=10$}\\
\hline
%&40   &1.61E-01 & -     &1.62E-01 & -    &2.68E-01 & -    &1.20E+00 &-  \\
&80   &9.68E-02&    -  &1.05E-01&    -    &1.73E-01&    -& 1.06E+00&    -\\
(\ref{imex:1})&160  &4.89E-02&    0.99& 5.34E-02&    0.98    &8.99E-02&    0.95& 5.89E-01&    0.85\\
with &320  &1.14E+00&   -4.54&2.70E-02&    0.99   &4.60E-02&    0.97& 3.20E-01&    0.88\\
$k=0$&640  &1.54E+00&   -0.43&1.35E-02&    0.99   &2.34E-02&    0.98& 1.71E-01&    0.91\\
&1280  &1.44E+00&    0.09&6.80E-03&    1.00    &1.18E-02&    0.99& 8.93E-02&    0.93\\
\hline
  \multicolumn{2}{c|}{}    & \multicolumn{2}{|c|}{$0.5\sin^2 1-0.1$}
&\multicolumn{2}{c|}{$0.5\sin^2 1$} &\multicolumn{2}{c|}{$a_0=1$}
&\multicolumn{2}{c}{$a_0=10$}\\
\hline
%&40   &8.72E+01 & -     &3.52E-03 & -    &7.65E-03 & -    &3.11E-01 &-  \\
&80    & 2.41E+01&    -   &1.01E-03&    - &3.41E-03&    -& 1.49E-01&    -\\
(\ref{imex:2}) &160  &3.67E+01&   -0.61 & 2.58E-04&    1.97   &9.05E-04&    1.91& 5.05E-02&    1.56\\
with&320  & 7.41E+01&   -1.01 &6.55E-05&    1.98  &2.34E-04&    1.95& 1.52E-02&    1.73\\
$k=1$&640  &1.52E+02&   -1.04 &  1.65E-05&    1.99 &5.97E-05&    1.97&4.22E-03&    1.85\\
&1280  & 3.34E+02&   -1.13 & 4.14E-06&    1.99  &1.51E-05&    1.99&  1.12E-03&    1.92\\
\hline
 \multicolumn{2}{c|}{}        & \multicolumn{2}{|c|}{$a_0=0.5\sin^2 1$}
&\multicolumn{2}{c|}{$a_0=0.54\sin^2 1$} &\multicolumn{2}{c|}{$a_0=1$}
&\multicolumn{2}{c}{$a_0=10$}\\
\hline
&80   &4.01E-05 & -     &4.27E-05 & -    &1.87E-04 & -    &4.11E-02 &-  \\
(\ref{imex:3})&160   &5.81E-06&    2.79 & 6.26E-06&    2.77   &2.98E-05&    2.65& 8.28E-03&    2.31\\
with &320  &1.01E+03&  -27.37& 8.78E-07&    2.83    &4.39E-06&    2.76& 1.38E-03&    2.59\\
$k=2$&640  &1.88E+03&   -0.90& 1.16E-07&    2.92   &6.01E-07&    2.87& 2.03E-04&    2.77\\
&1280  &3.25E+03&   -0.79& 1.49E-08&    2.95   &7.92E-08&    2.93& 2.78E-05&    2.87\\
\hline
\end{tabular}
\end{center}
\label{table3}
\end{table}

\medskip

\noindent{\emph{Example 2}}. To test the efficiency of the proposed methods
for problems with large variation of diffusion coefficients, we consider
the diffusion equation
$$u_t=(a(x) u_x)_x+f(x,t)$$
 with the same exact solution as Example 1.
We will consider $a(x)=1+b\sin^2(x)$ for  $b=10,100$ and $1000$.
Obviously, the diffusion coefficient varies from 1 to $1+b$,
and the variation is larger if $b$ is larger.

The choice of $a_0$ and time step in different situations
are given in Table \ref{time-step}. We see that the first and
second order schemes are stable if $a_0 \ge 0.5(1+b)$ and
the third order scheme is stable if $a_0 \ge 0.54(1+b)$.
The numerical results are listed in Table \ref{table4},
from which we can observe optimal orders of accuracy
of the proposed schemes. We also note that,
small mesh size and small time step are required to observe
optimal error accuracy for large $b$.

\begin{table}[h]
\caption{The constant $a_0$ and time step taken in the experiments.}
\centering
\footnotesize
\renewcommand{\arraystretch}{1.3}
\begin{tabular}{|c|c|c|c|}
\hline
\backslashbox{scheme}{$b$} & 10 & 100 & 1000 \\
\hline
(\ref{imex:1}) with $k=0$ &$a_0=6, \dt=0.1h$ &$a_0=51, \dt=0.1h$ &$a_0=501, \dt=0.01h$    \\
\hline
(\ref{imex:2}) with $k=1$&$a_0=6, \dt=0.1h$ &$a_0=51, \dt=0.1h$ &$a_0=501, \dt=0.01h$    \\
\hline
(\ref{imex:3}) with $k=2$&$a_0=6, \dt=0.1h$ &$a_0=55, \dt=0.05h$ &$a_0=540, \dt=0.01h$    \\
\hline
\end{tabular}
\label{time-step}
\end{table}

\begin{table}[htp]
\caption{The $L^2$ errors and orders of accuracy for Example 2: $a(x)=1+b\sin^2(x)$.}
\begin{center}\footnotesize
\renewcommand{\arraystretch}{1.3}
\begin{tabular}{|c|c|cc|cc|cc|}
\hline
\multicolumn{2}{|c}{ } & \multicolumn{2}{|c|}{$b=10$}
&\multicolumn{2}{c|}{$b=100$} &\multicolumn{2}{c|}{$b=1000$}
\\
\hline
scheme &$N$            & $L^2$ error &  order    & $L^2$ error &  order
& $L^2$ error & order   \\
\hline
&80   &4.66E-02 & -     &8.66E-02 & -    &5.19E-02 & -    \\
(\ref{imex:1})&160   &2.33E-02 & 1.00  & 4.73E-02 & 0.87   &2.63E-02 & 0.98\\
with &320  &1.17E-02 & 1.00& 2.48E-02 & 0.93  &1.34E-02 & 0.97\\
$k=0$&640  &5.84E-03 & 1.00& 1.27E-02 & 0.96   &6.76E-03 & 0.99\\
&1280  &2.92E-03  &1.00& 6.46E-03 & 0.98   &3.40E-03 & 0.99\\
%&2560  &1.46E-03 & 1.00& 3.25E-03 & 0.99    &1.70E-03 & 1.00\\
\hline
&80   &6.92E-04 & -     &5.02E-03 & -    &2.06E-03 & -    \\
(\ref{imex:2})&160   &1.74E-04  &1.99  & 1.56E-03 & 1.69   &6.40E-04 & 1.69\\
with &320  &4.36E-05 & 2.00& 4.49E-04 & 1.79  &1.84E-04 & 1.80\\
$k=1$&640  &1.09E-05 & 2.00& 1.23E-04 & 1.87   &5.01E-05 & 1.88\\
&1280  &2.73E-06 & 2.00& 3.24E-05 & 1.92   &1.32E-05 & 1.93\\
%&2560  &6.83E-07 & 2.00& 8.33E-06 & 1.96    &3.39E-06 & 1.96\\
\hline
&80   &5.60E-06 & -     &1.58E-04 & -    &3.65E-04 & -    \\
(\ref{imex:3})&160   &7.40E-07 & 2.92  &2.79E-05&    2.50   & 7.41E-05&    2.30\\
with &320  &9.66E-08 & 2.94&4.36E-06&    2.68  &1.30E-05&    2.51\\
$k=2$&640  &1.23E-08 & 2.97&6.28E-07&    2.80  &2.04E-06&    2.67\\
&1280  &1.46E-09 & 3.08&8.35E-08&    2.91   &2.88E-07&    2.82\\
\hline
\end{tabular}
\end{center}
\label{table4}
\end{table}

\medskip

From the stability analysis in Section \ref{sec3} and the numerical experiments
in this subsection, we propose a guidance for the choice of $a_0$ for general
model $U_t=(a(U)U_x)_x$, that is,  $a_0 \ge \max\{a(u)\}/2$ for the first
and second order schemes, and $a_0 \ge 0.54\max\{a(u)\}$ for the third order
scheme, where $u$ is the numerical solution at the corresponding time level.
In our experiments, the LU factorization is used as the linear solver,
it will cost more computation to solve a linear system with different
coefficient matrix at each time level. Actually in practical computing, it is
not necessary to scan the maximum of $a(u)$ and adjust $a_0$ at every time step.
In the next two subsections, we will simulate the porous medium equation and the
high-field model, where we adjust $a_0$ after every 100 time steps.

\subsection{Numerical simulation to the porous medium equation}

To further validate the performance of the proposed schemes,
we consider the porous medium equation (PME)
\begin{equation} \label{eq:pme}
u_t = (u^m)_{xx},
\end{equation}
in which $m$ is a constant greater than one. This equation often occurs in nonlinear problems
of heat and mass transfer, combustion theory, and flow in porous media, where $u$ is either a
concentration or a temperature required to be non negative. We assume the initial solution $u_0(x)$ is
a bounded non negative continuous function, then (\ref{eq:pme}) can be written as
\begin{equation} \label{eq:pme'}
u_t = (a(u)u_x)_x,
\end{equation}
with $a(u)=m u^{m-1}$. It is a degenerate parabolic equation since $u$ may be $0$ at some points.
The LDG schemes coupled with the explicit third order
RK time marching for solving this kind of problems were studied in \cite{Zhang:pme},
 where a slope limiter was introduced to ensure the non negativity of the numerical solutions.

In this subsection, we present the numerical results given by the EIN-LDG schemes.
In all the following experiments, we adopt $k=2$ for the spatial discretization,
and the second order scheme (\ref{imex:2}) for the time discretization.
\textcolor[rgb]{0,0,1}
{The same slope limiter as in \cite{Zhang:pme} is adopted at each intermediate stage.
Thanks to the limiter, we can ensure the non negativity of numerical solutions, 
and thus can ensure the diffusion coefficient $a(u)$ is non negative.
Moreover, the physical meaning of $u$ can be maintained, and the possible
numerical oscillation near discontinuous interfaces can be eliminated.}
All the experiments are tested on uniform mesh with mesh size $h=0.02$,
the time step is $\dt=\Ocal(h)$. In the experiments of this subsection,
we adjust $a_0$ after every $100$ time steps, according to the maximum of $a(u)$. We take
$a_0 = \max{a(u)}/2$ at the corresponding time levels.

\medskip

\noindent \emph{Test 1}. Equation (\ref{eq:pme'}) with the Barenblatt solution
\begin{equation}
B_m(x,t) = t^{-s} \left[ \left( 1-\frac{s(m-1)}{2m}\frac{|x|^2}{t^{2s}}\right)_{+}\right]^{1/(m-1)},
\end{equation}
where $u_{+}=\max\{u,0\}$ and $s=1/(m+1)$.
We begin the computation from $t=1$ in order to avoid the singularity of the Barenblatt solution near
$t=0$. The boundary condition is $u(\pm 6, t) = 0$ for $t\ge 1$.
We plot in Figure \ref{fig1} the numerical results for $m=2,3,5,8$ at $t=2$.
From this figure, we see that our scheme can simulate the Barenblatt solution
accurately and sharply, without noticeable oscillations near the interface.

 \begin{figure}[htp]
\centering
\subfigure[$m=2$]{
\begin{minipage}[b]{0.45\textwidth}
   \centering
    \includegraphics[width=3in]{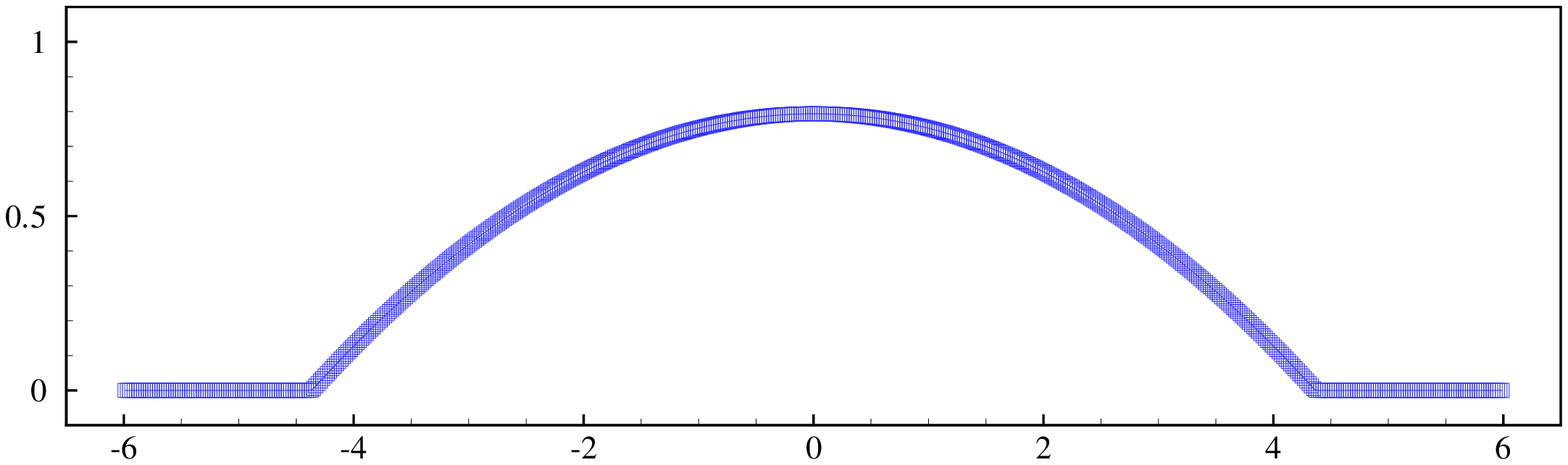}
  \end{minipage}}%
\subfigure[$m=3$]{
\begin{minipage}[b]{0.45\textwidth}
   \centering
    \includegraphics[width=3in]{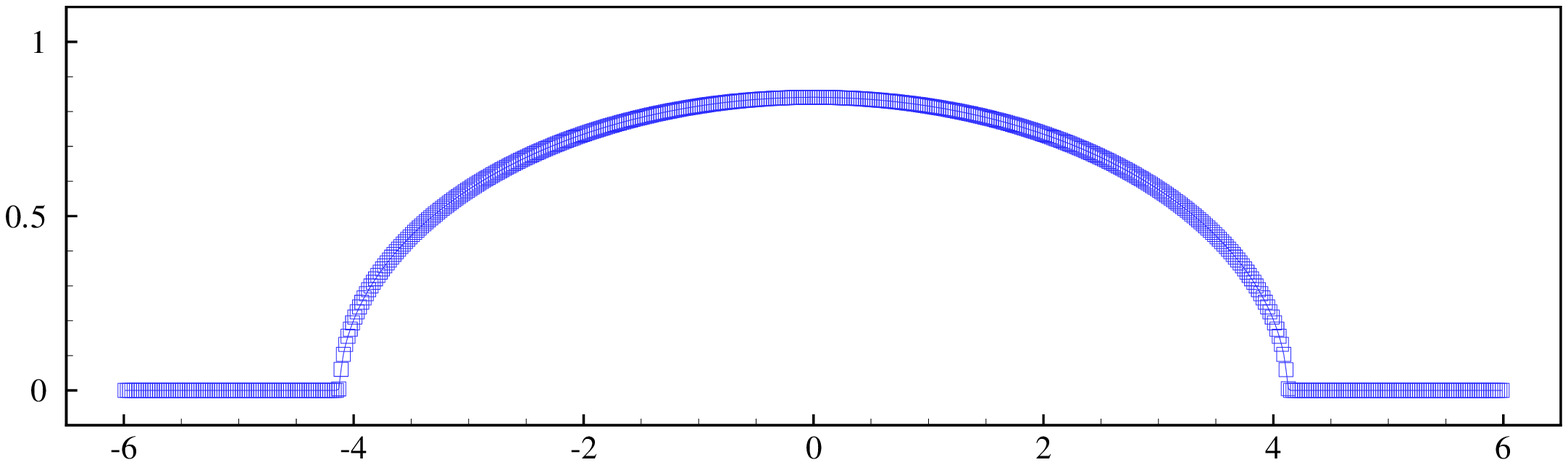}
  \end{minipage}}
\subfigure[$m=5$]{
\begin{minipage}[b]{0.45\textwidth}
   \centering
    \includegraphics[width=3in]{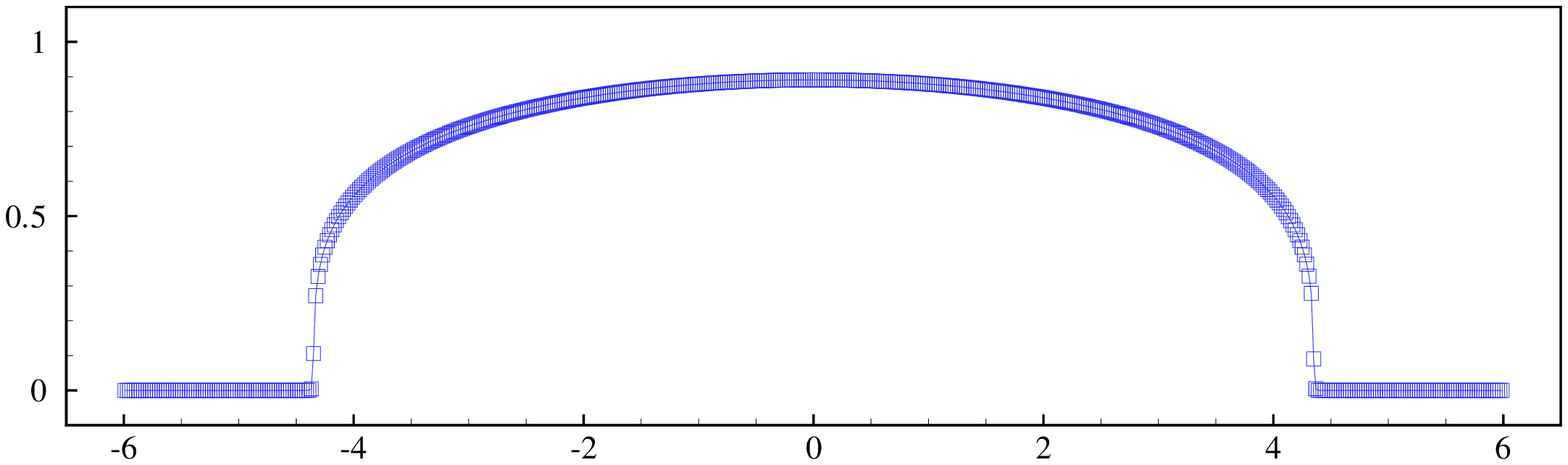}
  \end{minipage}}%
\subfigure[$m=8$]{
\begin{minipage}[b]{0.45\textwidth}
   \centering
    \includegraphics[width=3in]{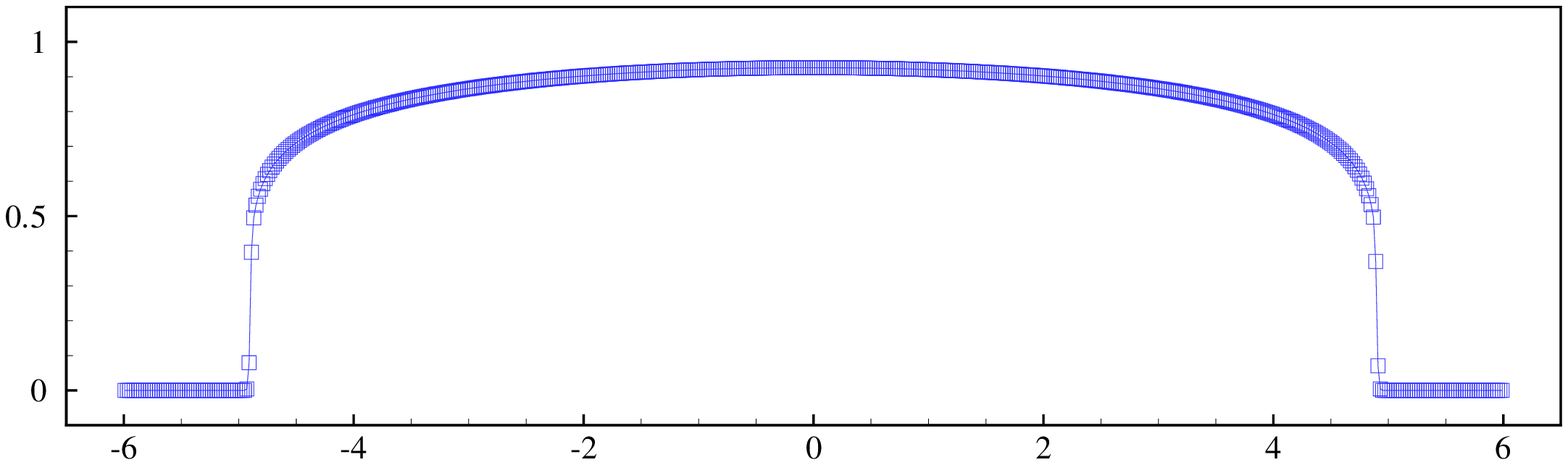}
  \end{minipage}}
\caption{Numerical results for the Barenblatt solution: $t=2$.}
\label{fig1}
\end{figure}

\smallskip

\noindent{\emph{Test 2}.} The collision of two-Box solutions with the same or different
heights. If the variable $u$ is regarded as the temperature, this model can be used to describe
how the temperature changes when two hot spots are suddenly put in the computation domain.
In Figure \ref{fig2} we plot the evolution of the numerical solution for the PME with $m = 5$. The
initial condition is the two-Box solution with the same height, namely
\begin{equation}
u_0(x) = \begin{cases} 1, & \mbox{if} \,\, x \in (-3.7,-0.7)\cup (0.7,3.7) \\
                           0, & \mbox{otherwise}
                           \end{cases}
\end{equation}
with the boundary condition $u(\pm 5.5,t)=0$ for all $t>0$.

In Figure \ref{fig3} we plot the evolution of the numerical solution for the PME with
parameter $m = 8$. The
initial condition is defined as
\begin{equation}
u_0(x) = \begin{cases} 1, & \mbox{if} \,\, x \in (-4,-1), \\
                       1.5, & \mbox{if} \,\, x\in (0,3),  \\
                       0  , & \mbox{otherwise}
                           \end{cases}
\end{equation}
with the boundary condition $u(\pm 6,t)=0$ for all $t>0$.

From these simulations, we can see an analogous evolution whether the heights of the
two boxes in the initial condition are the same or not. Two-Box solutions first move outward
independently before the collision, then they join each other to make the temperature
smooth, and finally the solution becomes almost constant in the common support.

 \begin{figure}[htp]
\centering
\subfigure[t=0.0]{
\begin{minipage}[b]{0.45\textwidth}
   \centering
    \includegraphics[width=3in]{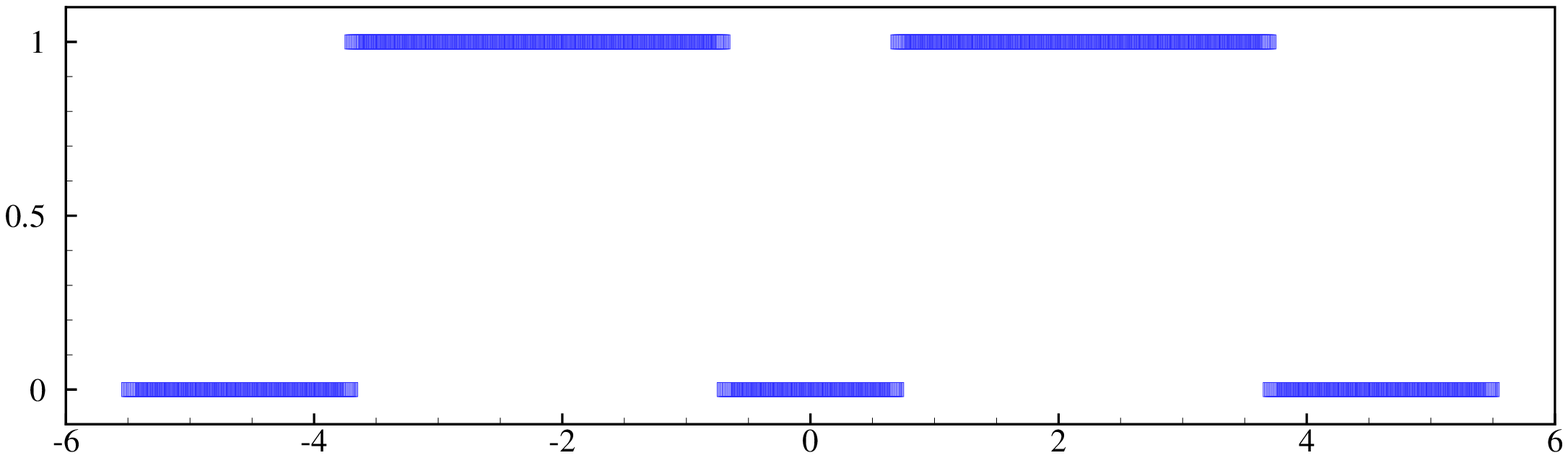}
  \end{minipage}}%
\subfigure[t=0.3]{
\begin{minipage}[b]{0.45\textwidth}
   \centering
    \includegraphics[width=3in]{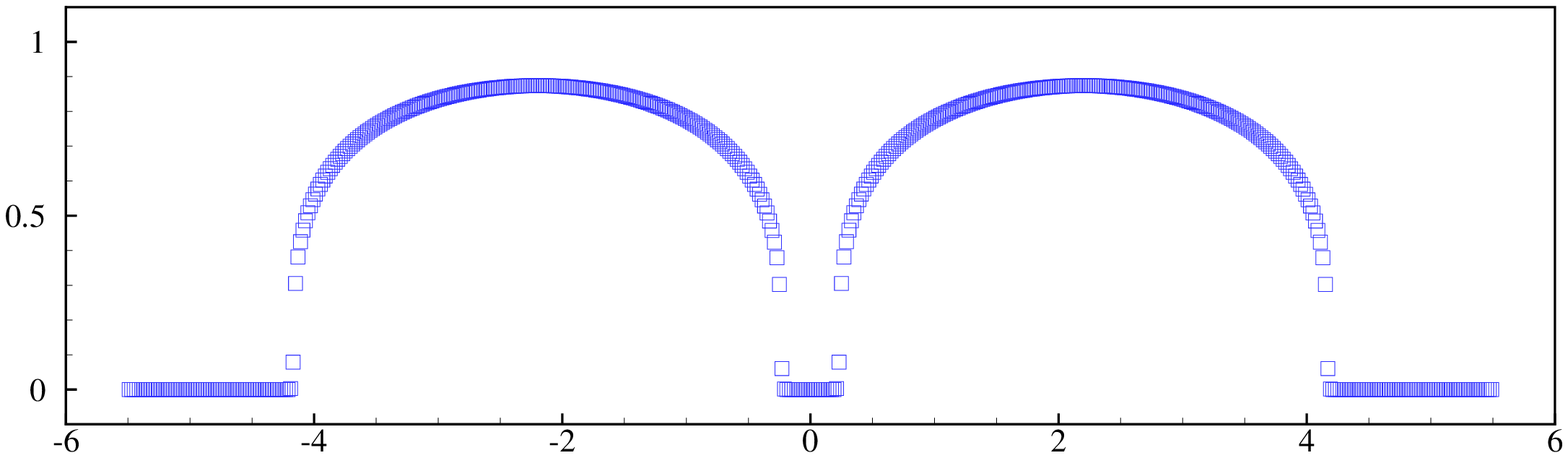}
  \end{minipage}}
\subfigure[t=0.6]{
\begin{minipage}[b]{0.45\textwidth}
   \centering
    \includegraphics[width=3in]{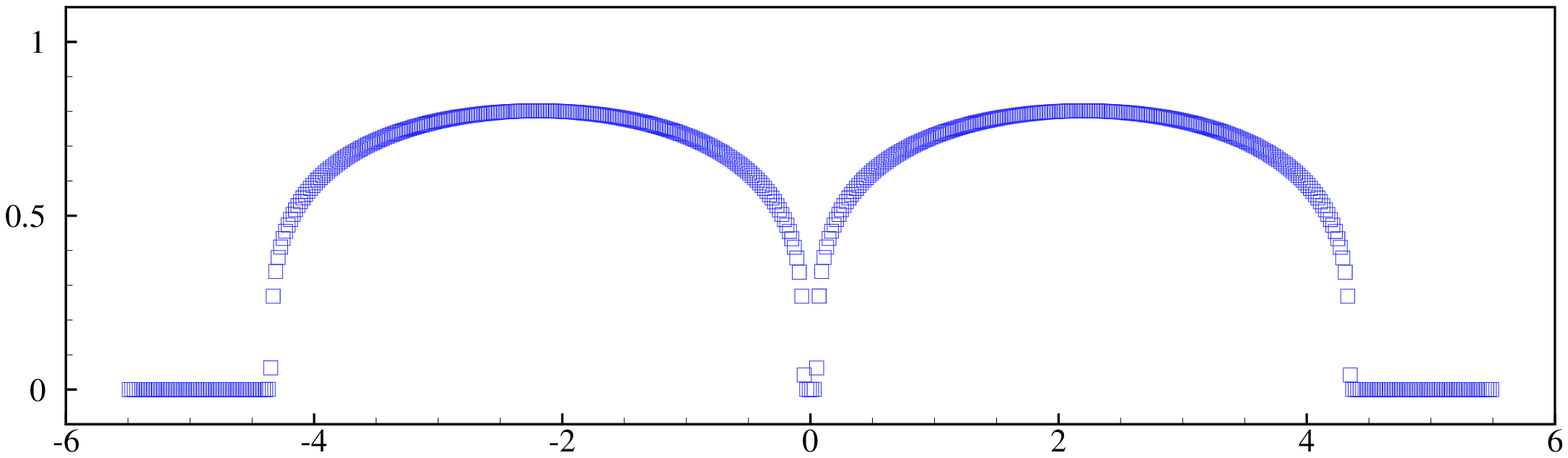}
  \end{minipage}}%
\subfigure[t=0.9]{
\begin{minipage}[b]{0.45\textwidth}
   \centering
    \includegraphics[width=3in]{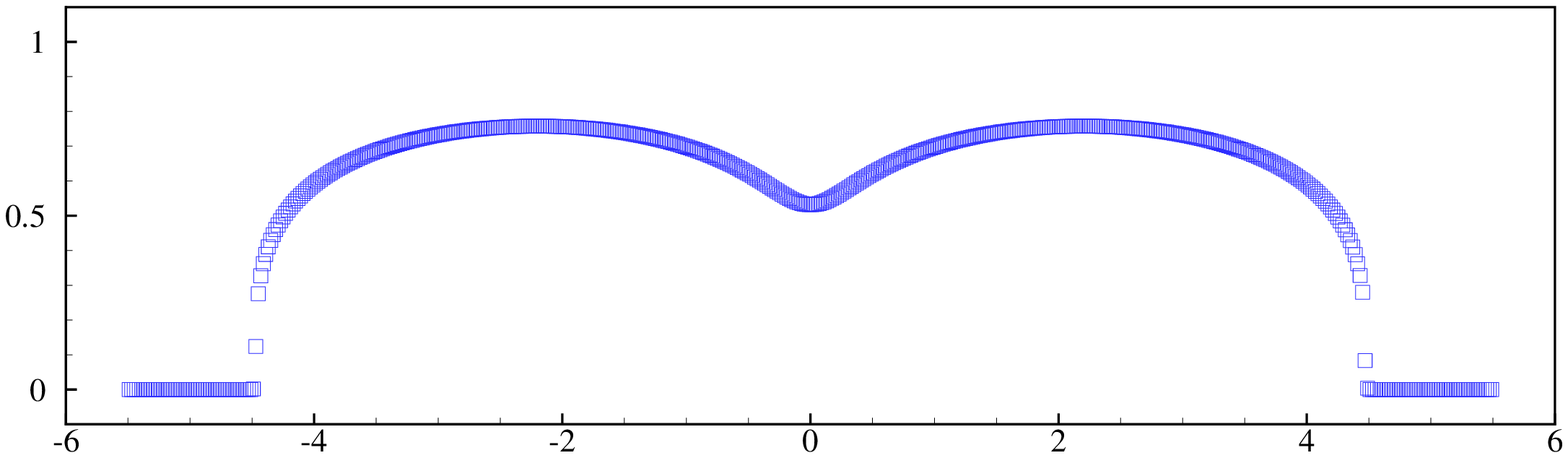}
  \end{minipage}}
\subfigure[t=1.2]{
\begin{minipage}[b]{0.45\textwidth}
   \centering
    \includegraphics[width=3in]{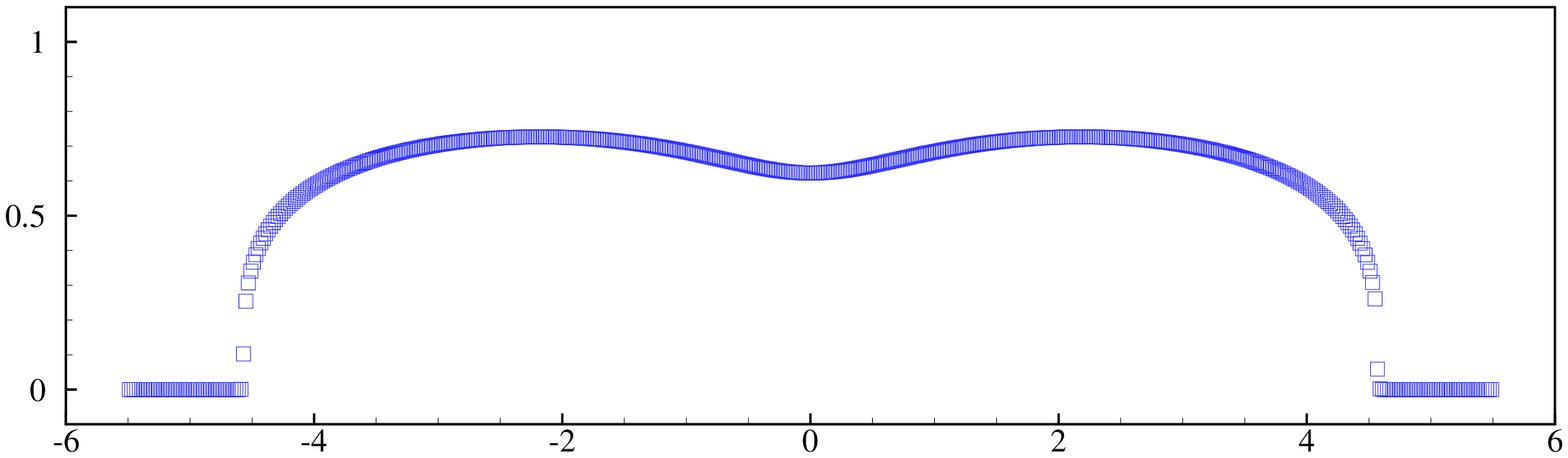}
  \end{minipage}}%
\subfigure[t=1.5]{
\begin{minipage}[b]{0.45\textwidth}
   \centering
    \includegraphics[width=3in]{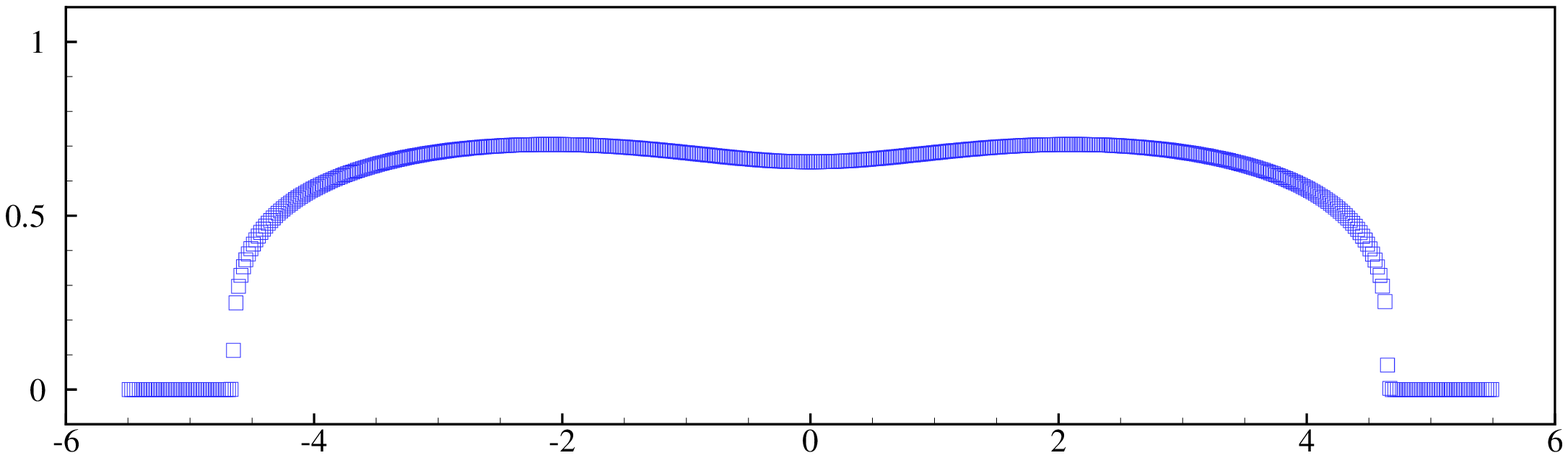}
  \end{minipage}}
\caption{Collision of the two-Box solution with the same height.}
\label{fig2}
\end{figure}

 \begin{figure}[htp]
\centering
\subfigure[t=0.0]{
\begin{minipage}[b]{0.45\textwidth}
   \centering
    \includegraphics[width=3in]{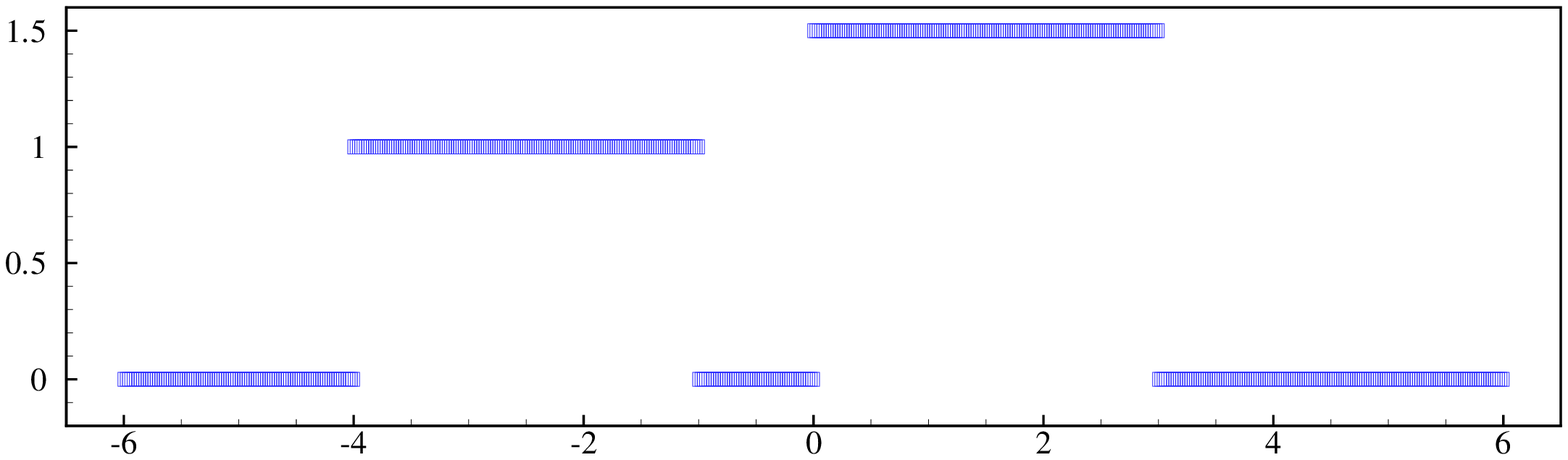}
  \end{minipage}}%
\subfigure[t=0.05]{
\begin{minipage}[b]{0.45\textwidth}
   \centering
    \includegraphics[width=3in]{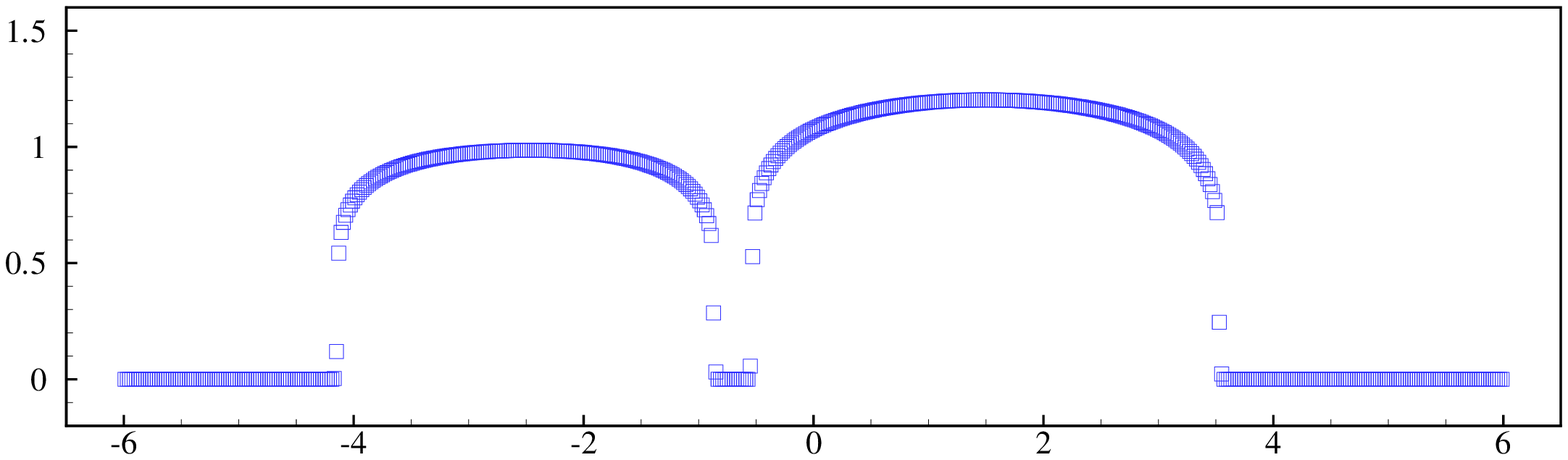}
  \end{minipage}}
\subfigure[t=0.08]{
\begin{minipage}[b]{0.45\textwidth}
   \centering
    \includegraphics[width=3in]{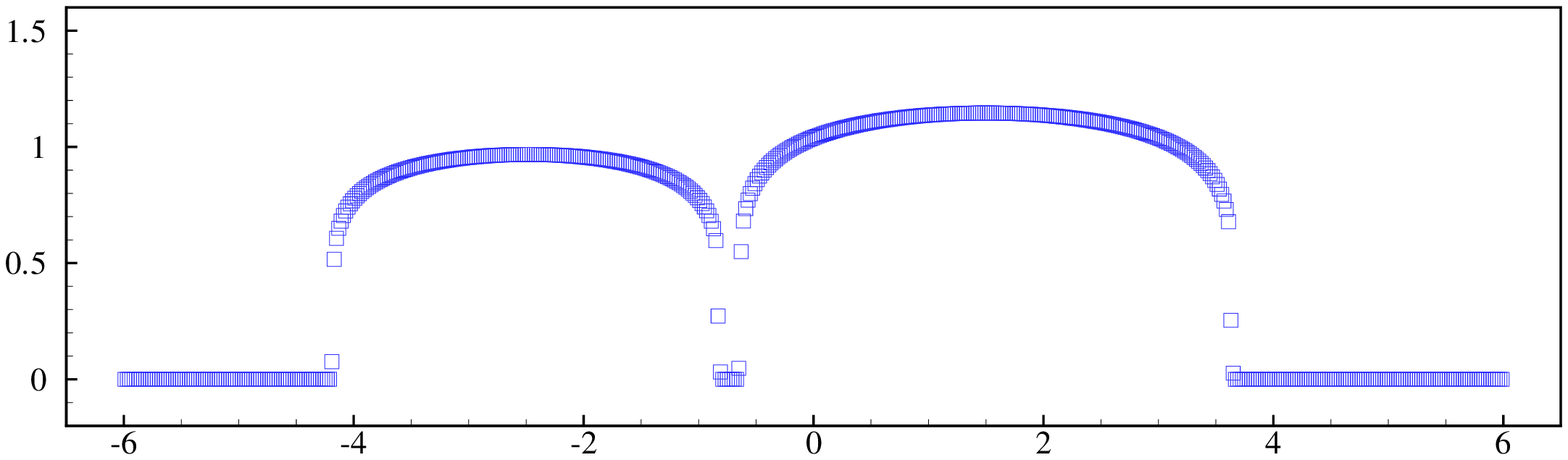}
  \end{minipage}}%
\subfigure[t=0.11]{
\begin{minipage}[b]{0.45\textwidth}
   \centering
    \includegraphics[width=3in]{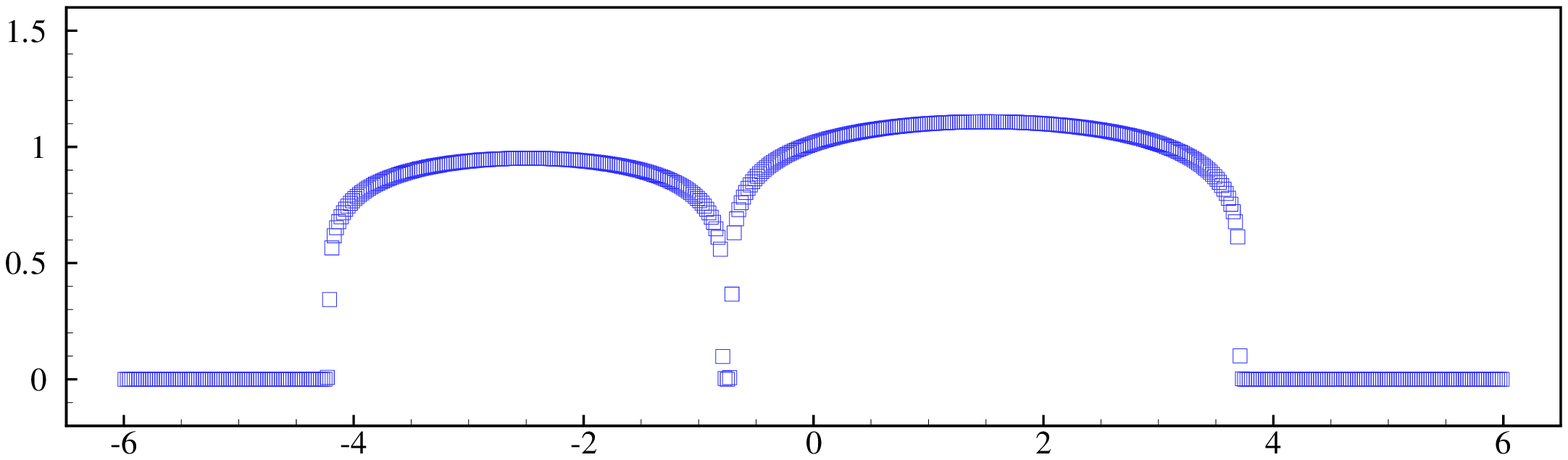}
  \end{minipage}}
\subfigure[t=0.14]{
\begin{minipage}[b]{0.45\textwidth}
   \centering
    \includegraphics[width=3in]{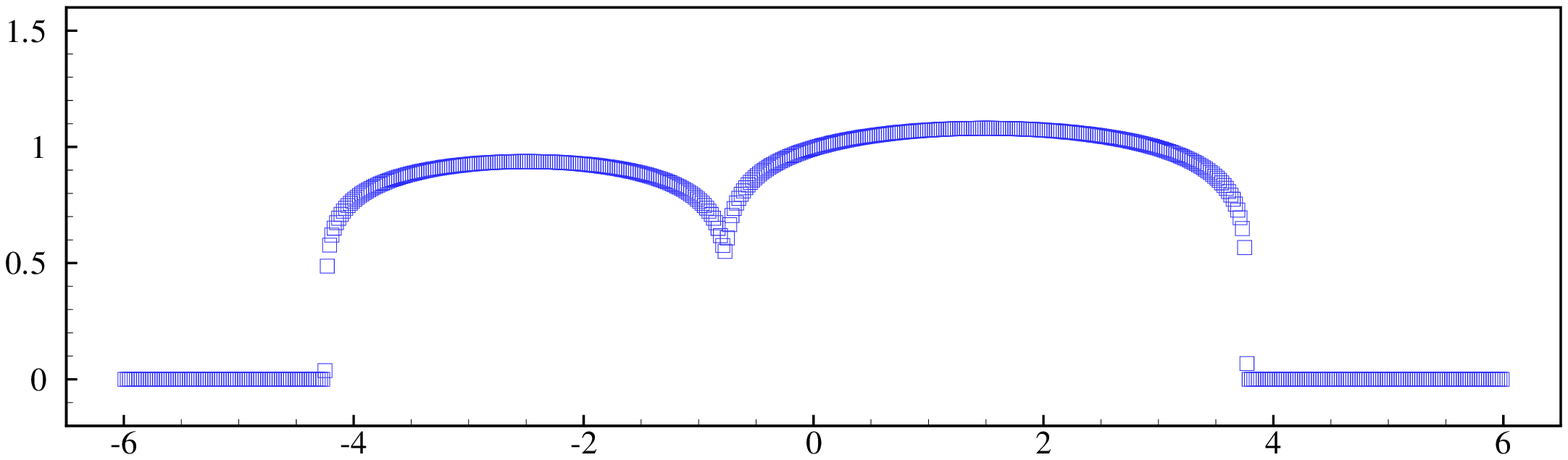}
  \end{minipage}}%
\subfigure[t=0.17]{
\begin{minipage}[b]{0.45\textwidth}
   \centering
    \includegraphics[width=3in]{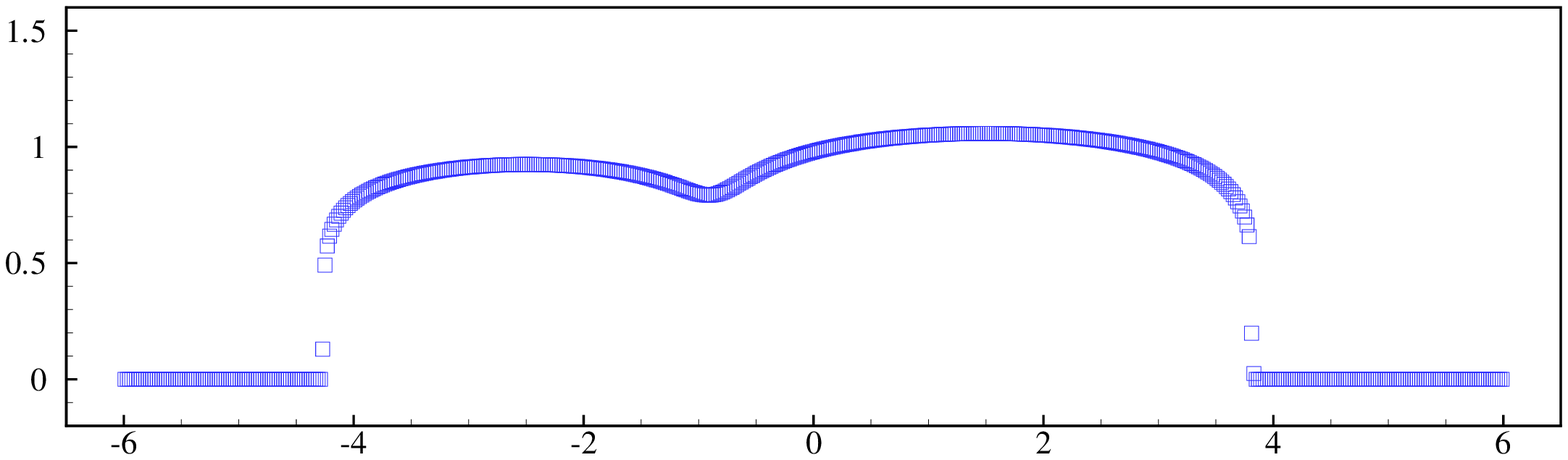}
  \end{minipage}}
\subfigure[t=0.20]{
\begin{minipage}[b]{0.45\textwidth}
   \centering
    \includegraphics[width=3in]{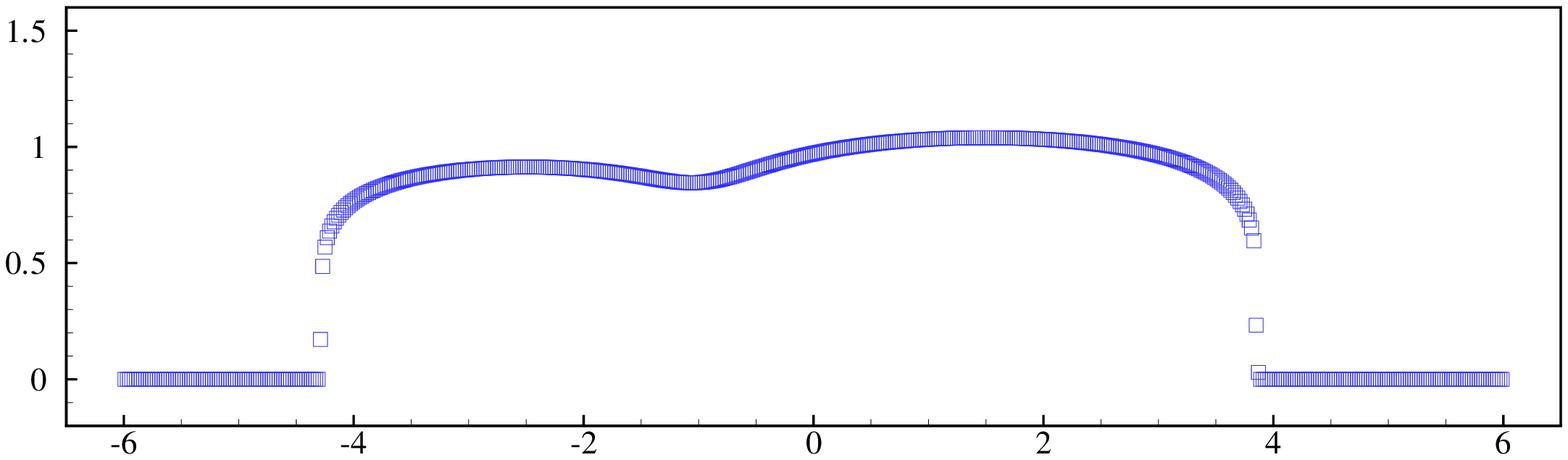}
  \end{minipage}}%
\subfigure[t=0.23]{
\begin{minipage}[b]{0.45\textwidth}
   \centering
    \includegraphics[width=3in]{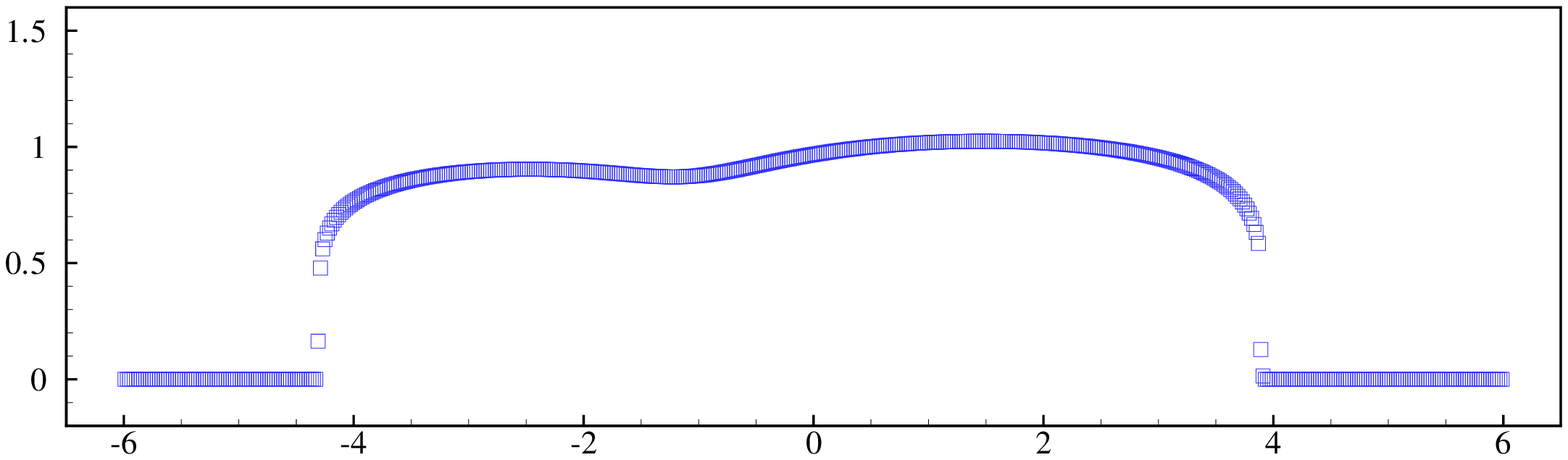}
  \end{minipage}}
\subfigure[t=0.50]{
\begin{minipage}[b]{0.45\textwidth}
   \centering
    \includegraphics[width=3in]{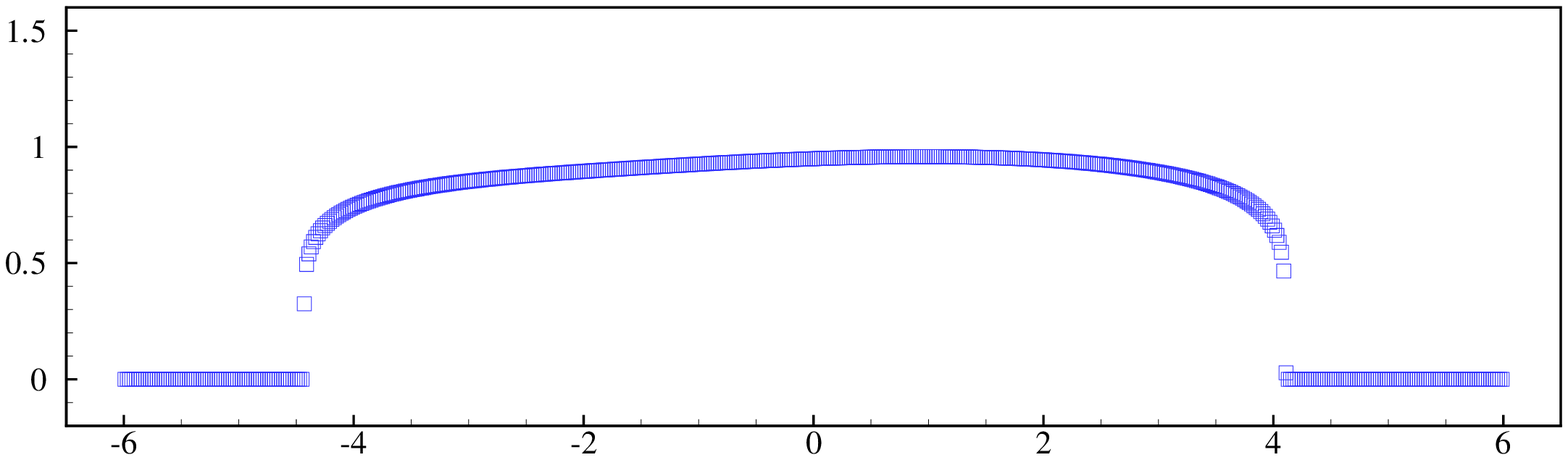}
  \end{minipage}}%
\subfigure[t=1.00]{
\begin{minipage}[b]{0.45\textwidth}
   \centering
    \includegraphics[width=3in]{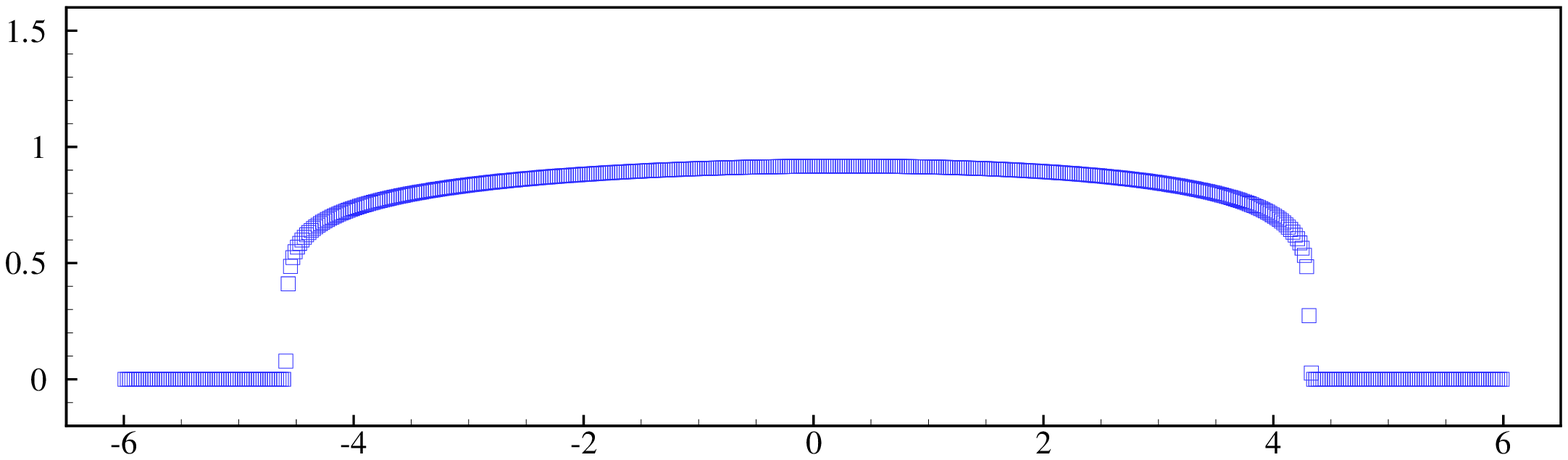}
  \end{minipage}}
\caption{Collision of the two-Box solution with different heights.}
\label{fig3}
\end{figure}

\smallskip
\noindent{\emph{Test 3}}. To test the waiting time phenomenon \cite{Angenent}, i.e,
the interface of the support does not move outward until the waiting time,
  we consider the PME with $m = 8$. The initial condition
is defined as a fast-varying solution, namely,
\begin{equation}
u_0(x) = \begin{cases} \cos x, & \mbox{if} \,\, x \in (-\pi/2,\pi/2)\\
                           0, & \mbox{otherwise}
                           \end{cases}
\end{equation}
with the boundary condition $u(\pm \pi, t)=0$ for all $t>0$.
We plot in Figure \ref{fig4} the evolution of the numerical solutions.
We observe that the interface begins to  move outward around $t=1.4$,
before that, the interface does not move outward, which verifies the
waiting-time phenomenon.

\smallskip
From the above experiments, we see that our scheme can simulate the
PME accurately. The main advantage is the fact that larger time steps can be chosen
compared with the explicit time discretization methods, where
$\dt=\Ocal(h^2)$ is required.

\begin{figure}[htp]
\centering
\subfigure[t=0.0]{
\begin{minipage}[b]{0.45\textwidth}
   \centering
    \includegraphics[width=3in]{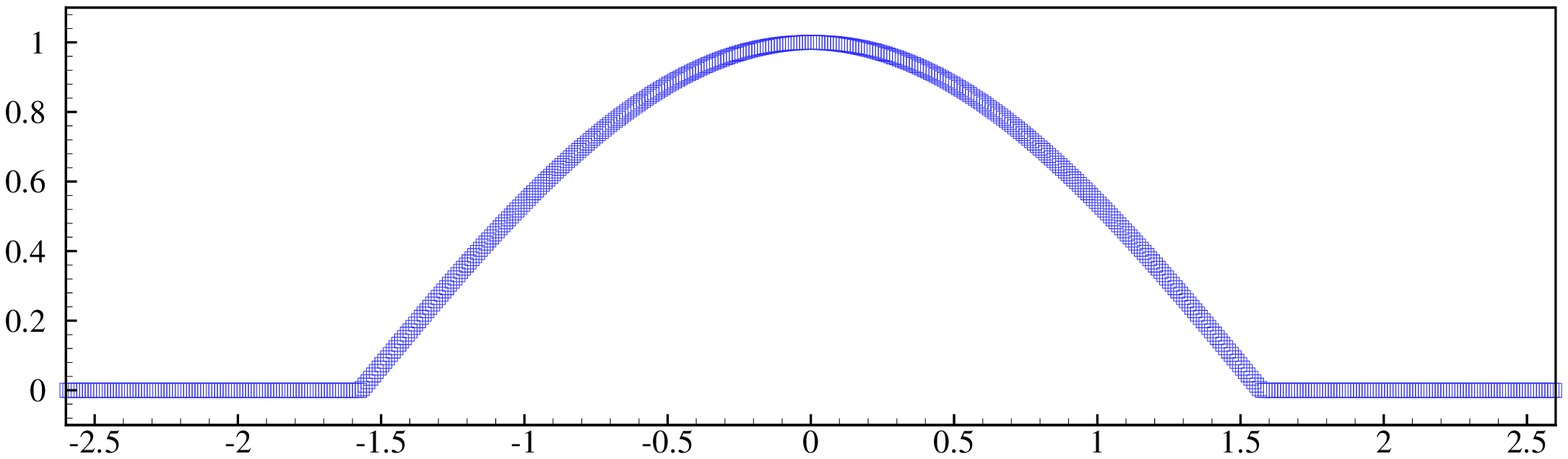}
  \end{minipage}}%
\subfigure[t=0.2]{
\begin{minipage}[b]{0.45\textwidth}
   \centering
    \includegraphics[width=3in]{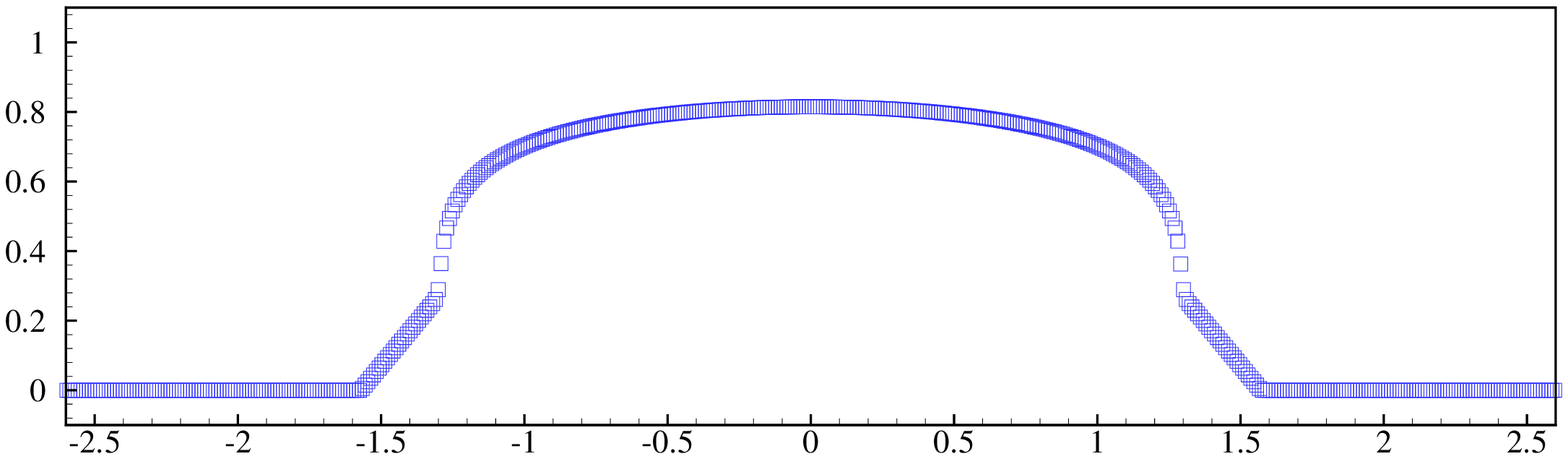}
  \end{minipage}}
\subfigure[t=0.4]{
\begin{minipage}[b]{0.45\textwidth}
   \centering
    \includegraphics[width=3in]{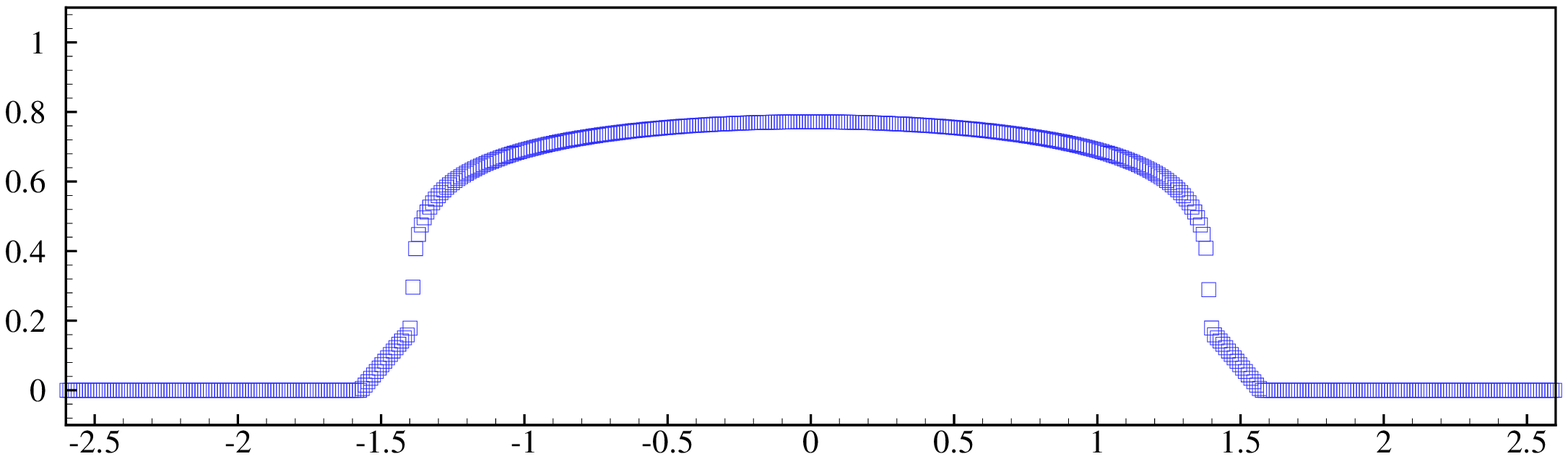}
  \end{minipage}}%
\subfigure[t=0.6]{
\begin{minipage}[b]{0.45\textwidth}
   \centering
    \includegraphics[width=3in]{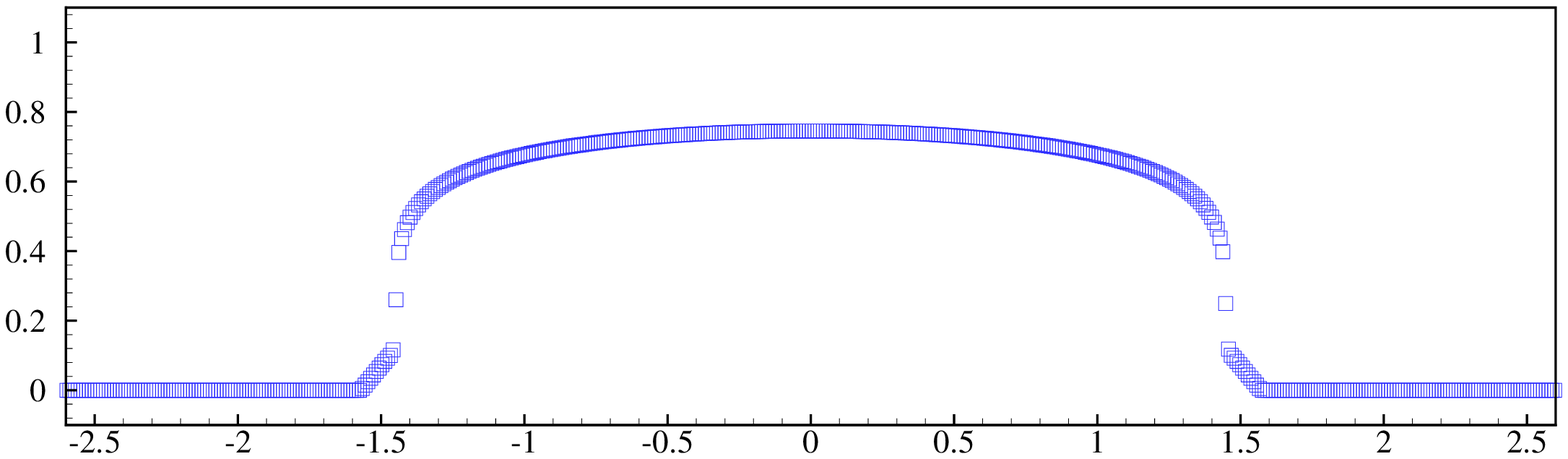}
  \end{minipage}}
\subfigure[t=0.8]{
\begin{minipage}[b]{0.45\textwidth}
   \centering
    \includegraphics[width=3in]{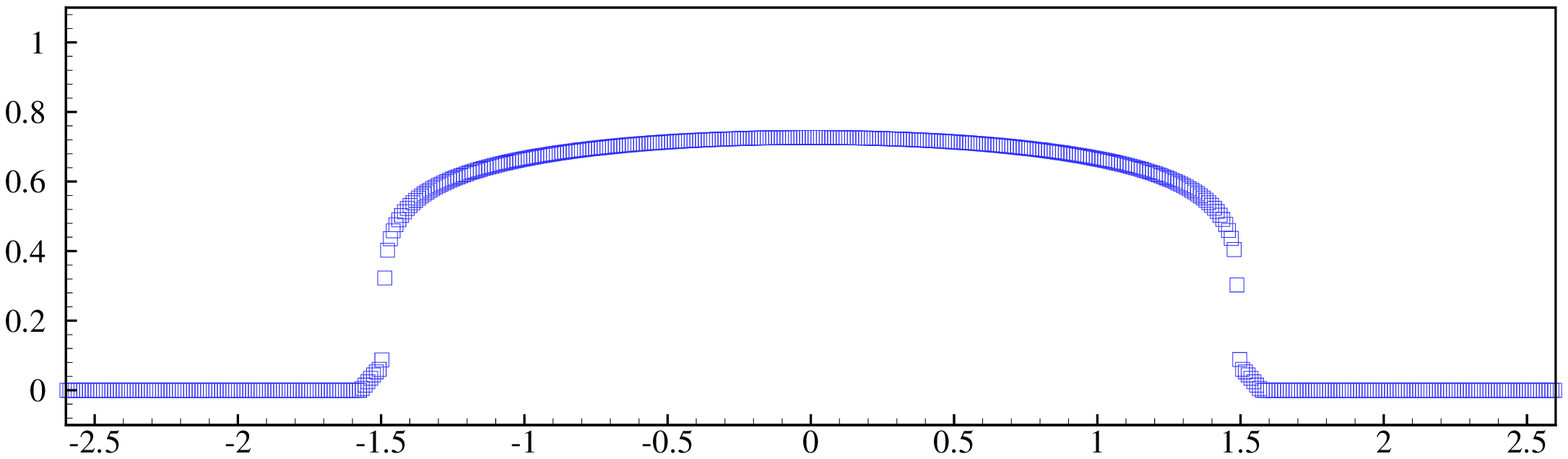}
  \end{minipage}}%
\subfigure[t=1.0]{
\begin{minipage}[b]{0.45\textwidth}
   \centering
    \includegraphics[width=3in]{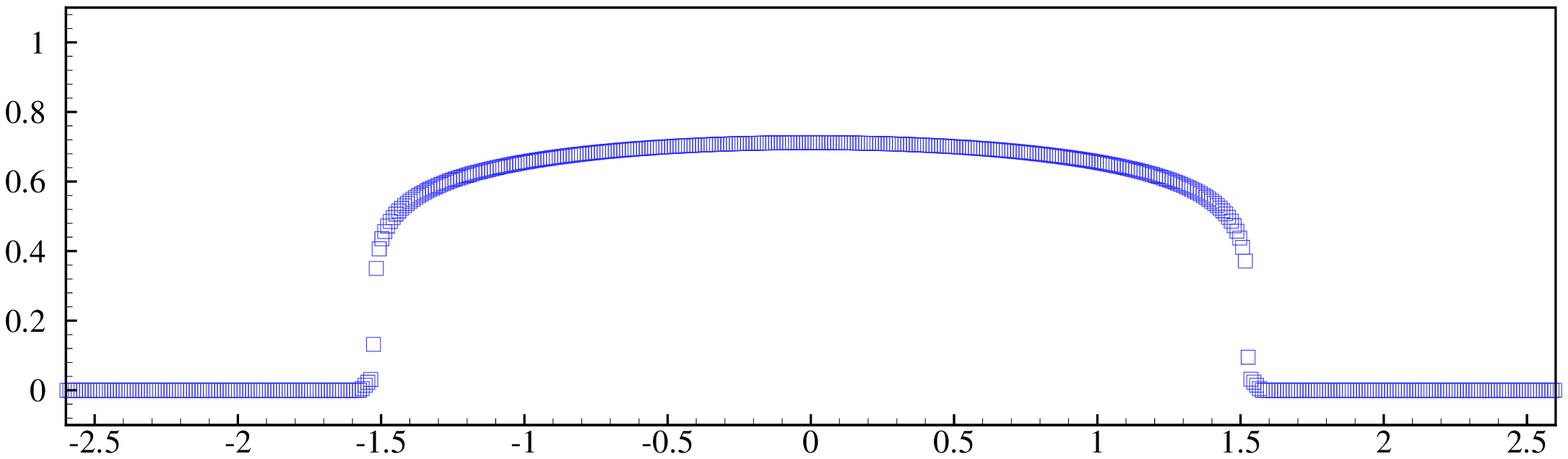}
  \end{minipage}}
\subfigure[t=1.2]{
\begin{minipage}[b]{0.45\textwidth}
   \centering
    \includegraphics[width=3in]{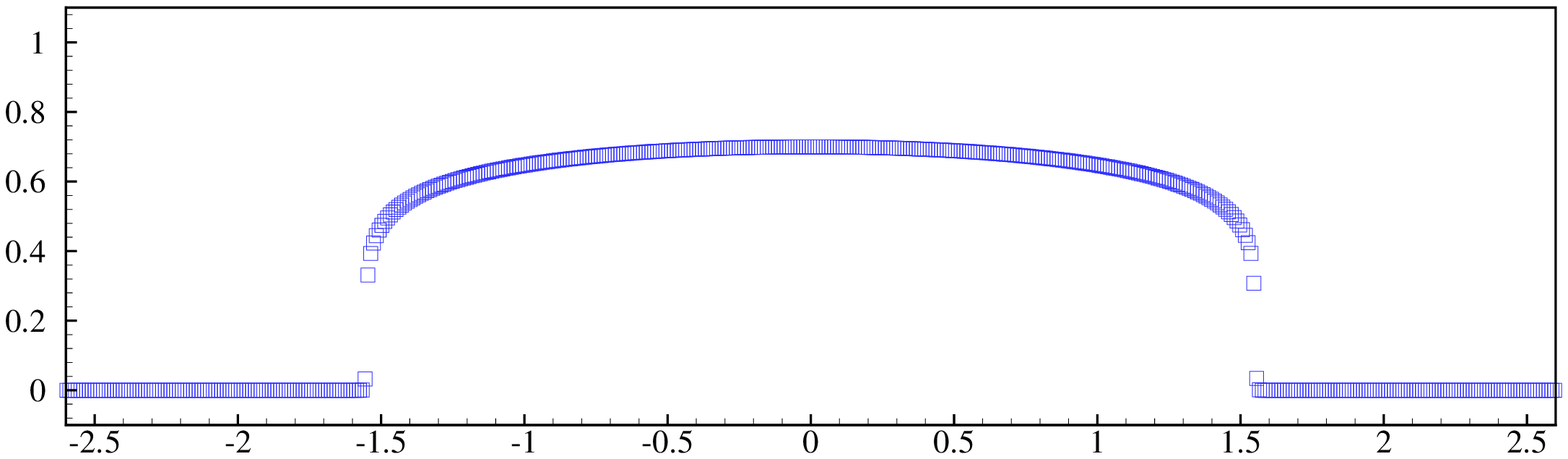}
  \end{minipage}}%
\subfigure[t=1.4]{
\begin{minipage}[b]{0.45\textwidth}
   \centering
    \includegraphics[width=3in]{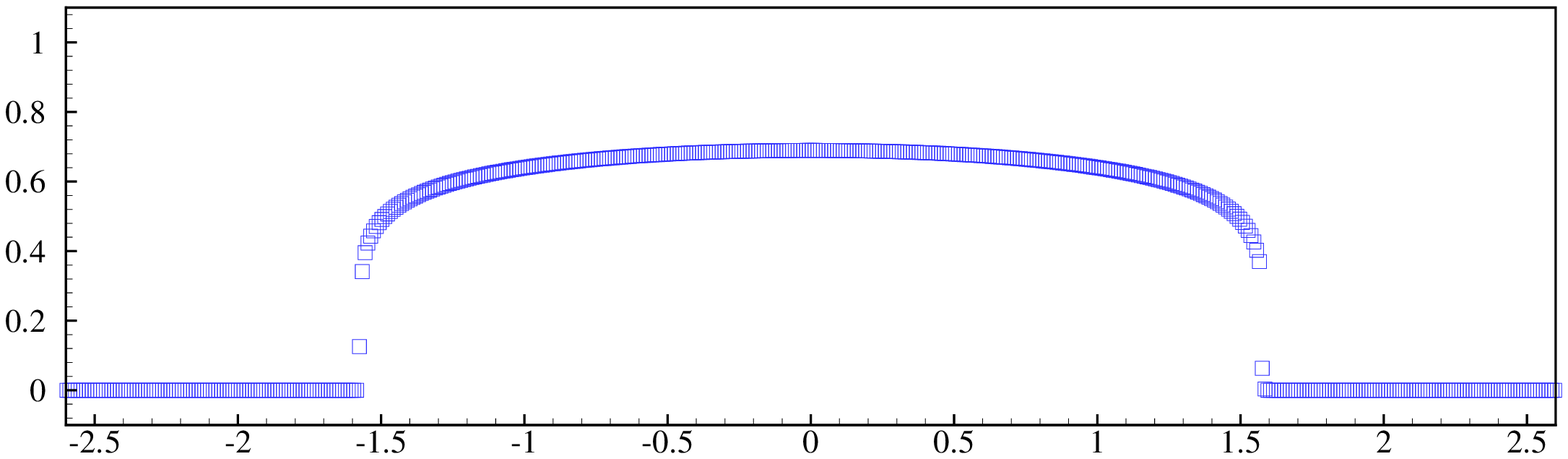}
  \end{minipage}}
\subfigure[t=1.6]{
\begin{minipage}[b]{0.45\textwidth}
   \centering
    \includegraphics[width=3in]{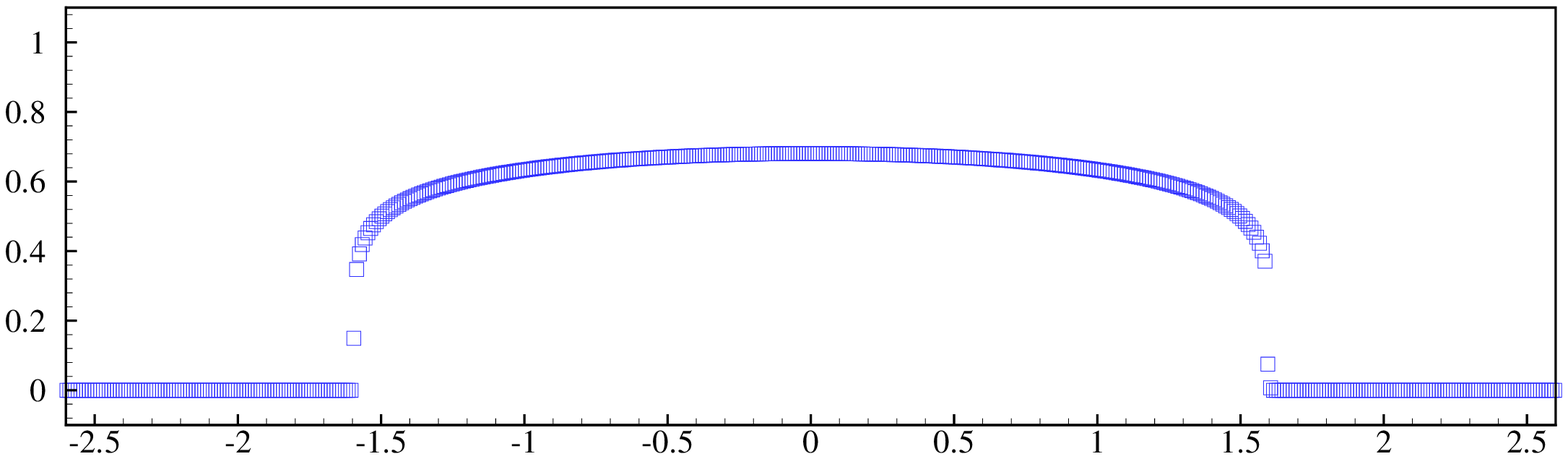}
  \end{minipage}}%
\subfigure[t=1.8]{
\begin{minipage}[b]{0.45\textwidth}
   \centering
    \includegraphics[width=3in]{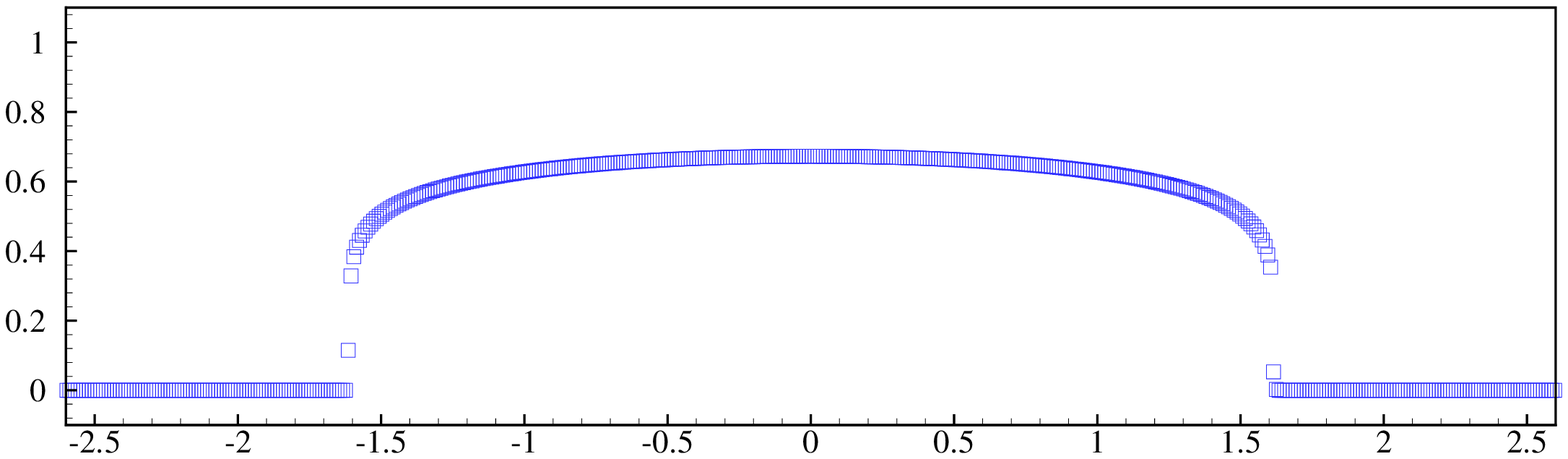}
  \end{minipage}}
\caption{Waiting-time phenomenon.}
\label{fig4}
\end{figure}

\subsection{Numerical simulation to the high-field model}
In this subsection, we apply the proposed scheme to the one-dimensional
high-field (HF) model \cite{HF,HF1} in semiconductor device simulations, which is
a convection-diffusion system coupled with a Poisson potential equation.
The notations for the model are only valid in this subsection.
The HF model is described by the following equation
\begin{equation} \label{eq:hf}
n_t + J_x = 0,  \qquad  x \in (0,0.6)
\end{equation}
where
\[ J = J_{hyp} + J_{vis},
\]
and
\begin{align*}
J_{hyp} =&\, -\mu n E + \gamma \mu \left(\frac{e}{\varepsilon}\right)n(-\mu n E + \omega), \\
J_{vis} =&\, -\gamma [n(\theta+2\mu^2 E^2)]_x + \gamma \mu E (\mu n E)_x.
\end{align*}
In the HF model, the unknown variable $n$ is the electron concentration,
$E=-\phi_x$ is the electric field,  and $\phi$ is the electric potential
which is given by the Poisson equation
\begin{equation}\label{eq:Poisson}
\phi_{xx} = \frac{e}{\varepsilon} (n-n_d),
\end{equation}
with $n_d$ being a given doping (also the initial condition for $n$).
The boundary conditions for $n$ and $E$ are periodic, and for $\phi$ is Dirichlet
boundary condition which will be given later.

In the above model, the parameter $\mu$ is the mobility, $e$ is the electron charge,
$\varepsilon$ is the dielectric permittivity, $\omega=(\mu n E)|_{x=0}$ is taken
to be a constant, $\gamma=\frac{m \mu}{e}$ is the relaxation parameter, with $m$
being the electron effective mass, and $\theta = \frac{k}{m}T_0$, where $k$ is
the Boltzmann constant and $T_0$ is the lattice temperature.

The LDG method has been applied to solve problem (\ref{eq:hf}) in \cite{Liu2}
by Liu and Shu, where they used the third order explicit RK method in
the time discretizaiton. In their later work \cite{Liu3}, an IMEX-LDG method
was adopted to solve the drift-diffusion model, for which the coefficient
of diffusion is constant. In \cite{Liu3}, the IMEX-LDG method shows good
efficiency compared with explicit methods. Here the diffusion of HF model
is nonlinear, we will use the proposed EIN-LDG scheme.

For the convenience of adopting  EIN-LDG scheme, we rewrite the
HF model (\ref{eq:hf}) as
\begin{equation}
n_t + \left( -\mu n E -\gamma \mu^2 \frac{e}{\varepsilon}n^2 E +
\gamma \mu \frac{e}{\varepsilon} \omega n - 3\gamma \mu E n (\mu E)_x \right)_x
-[(\gamma \theta+ \gamma \mu^2 E^2)n_x]_x=0.
\end{equation}
Using $E_x=-\frac{e}{\varepsilon}(n-n_d)$, we can write the equation as
\begin{equation} \label{13}
n_t + f(n,E)_x - (a(E)n_x)_x=0,
\end{equation}
where
\begin{align*}
f(n,E) =&\, \gamma \mu^2 \frac{e}{\varepsilon} nE (2n-3n_d)-\mu n E (1+3\gamma E \mu_x)
+ \gamma \mu \frac{e}{\varepsilon} \omega n ,\\
a(E) =&\, \gamma \theta + \gamma \mu^2 E^2.
\end{align*}
Then by adding  a term $a_0 n_{xx}$ on both sides of (\ref{13}) we get
\begin{equation} \label{14}
n_t + \underbrace{f(n,E)_x - [(a(E)-a_0)n_x]_x}_{explicit} = \underbrace{a_0 n_{xx}}_{implicit},
\end{equation}
with periodic boundary conditions for $n$ and $E$,
where $a_0$ is a properly chosen positive constant.
We solve (\ref{14}) by the standard LDG scheme with
the third order IMEX scheme (\ref{imex:3}), where piecewise
quadratic polynomials space is adopted in spatial discretization,
Lax-Friedriches numerical flux and alternating numerical flux are
used for the convection and diffusion parts, respectively.
We treat the part on the left hand side explicitly, and the
part on the right hand side implicitly.

We point out that the potential equation (\ref{eq:Poisson}) is also
solved by the LDG method, i.e, finding $(\phi_h,\psi_h)\in V_h \times V_h$, such that
for any $(v,r)\in V_h \times V_h$, there holds
\begin{subequations} \label{15}
\begin{align}
(\frac{e}{\varepsilon}(n-n_d),v)_j=&\,-(\psi_h,v_x)_j + \widehat{\psi_h}_{\jfhalf} v_{\jfhalf}^- -
\widehat{\psi_h}_{\jbhalf} v_{\jbhalf}^+ , \\
 (\psi_h,r)_j =&\, -(\phi_h,r_x)_j + \widehat{\phi_h}_{\jfhalf} r_{\jfhalf}^- -
\widehat{\phi_h}_{\jbhalf} r_{\jbhalf}^+ ,
\end{align}
\end{subequations}
for $j=1,2,\cdots,N$, where we take the minimal dissipation numerical flux as in \cite{Cock:Dong},
specifically
\begin{equation*}
\widehat{\phi_h}_{\jfhalf}=\begin{cases} \phi_a, &  j=0,\\
                                         (\phi_h)_{\jfhalf}^-, & j=1,2,\cdots,N-1,\\
                                         \phi_b, & j=N.
                            \end{cases}
\end{equation*}
\begin{equation*}
\widehat{\psi_h}_{\jfhalf}=\begin{cases} (\psi_h)_{\jfhalf}^+, & j=0,1,\cdots,N-1,\\
                                         (\psi_h)_{N+\frac12}^- - \frac{1}{h}
                                         \left((\phi_h)_{N+\frac12}^- - \phi_b\right), & j=N.
                            \end{cases}
\end{equation*}
Here $\phi_a$ and $\phi_b$ are the given Dirichlet boundary conditions, $h$ is the mesh size.
The numerical approximation of electric field is given by $E_h=-\psi_h$.

\medskip

Next we simulate the HF model with the same parameters as in \cite{Liu2}.
The doping $n_d$ is  a piecewise-defined function in $[0.0.6]$,
$n_d=5\times 10^{17} \rm{cm}^{-3}$ in $[0,0.1]$ and $[0.5,0.6]$,
$n_d=2 \times 10^{15} \rm{cm}^{-3}$ in $[0.15,0.45]$,
and a smooth transition in between. The lattice temperature
is $T_0=300 \rm{K}$. The constants $k=0.138\times 10^{-4}$, $\varepsilon=11.7\times 8.85418$,
$e=0.1602$, $m=0.26\times 0.9109 \times 10^{-31} \rm{kg}$, and the mobility $\mu=0.0088\left(1+
\frac{14.2273}{1+\frac{n_d}{143200}}\right)$ in our units. The boundary conditions are given as follows:
$\phi_a=\frac{kT}{e}\ln(\frac{n_d}{n_i})$ at the left boundary, with $n_i=1.4\times 10^{10}\rm{cm}^{-3}$,
$\phi_b=\phi_a + v_{bias}$ with the voltage drop $v_{bias}=1.5$ at the right boundary for the potential;
$T=300 \rm{K}$ at both boundaries for the temperature; and $n=5\times 10^{17} \rm{cm}^{-3}$ at both
boundaries for the concentration.

In the simulations, we let $a_0=\max\{a(E_h)\}$ in (\ref{14}) and adjust it after every $100$ steps,
here $E_h=-\psi_h$ is solved from (\ref{15}). The code runs until the numerical solution converges to
the steady state, we use $\norm{n_h^{nt}-n_h^{nt-1}}_{L^1}< 10^{-6}$ as the criterion for stopping
computation, where $n_h$ is the numerical solution of the electron concentration $n$, and $nt$
is the number of time steps. \textcolor[rgb]{0,0,1}{The positivity limiter \cite{Zhang:pme} 
is not necessary for this example, since the minimum value of $n_h$ will not be below 0 due to
the initial setting of $n_d$ defined above.}

Table \ref{table5} and Table \ref{table6} show the time step, the number of time steps,
the numerical steady time, and the
CPU time to reach the steady state for the third order explicit RK LDG (EX-RK-LDG)
and the third order EIN-LDG methods
when we use $100$ mesh cells and $200$ mesh cells in $[0,0,6]$, respectively.
From these tables, we see that the proposed EIN-LDG scheme can take
much larger time steps compared with the explicit method, and hence it saves
in CPU time significantly.
%We also observe that the scheme is stable regardless of the choice of $h$
%(100 or 200 mesh cells).
%\textcolor[rgb]{0,0,1}
{On the other hand, due to the larger time step,
the numerical steady time for EIN-LDG scheme is greater than that for
the EX-RK-LDG scheme.}
Figure \ref{fig5} plots the simulation results of the HF model with
$200$ mesh cells, for both the EX-RK-LDG method and the EIN-LDG method.
It shows that the EIN-LDG method gives the same convergent results as
the explicit method.  The EIN-LDG scheme is thus a reliable and efficient
tool for the study of  models such as the HF model
to describe the correct physics.

\begin{table}[htp]
\caption{The time step $\dt$, the number of time steps $nt$,
the numerical steady time $t$,
and the CPU time to reach the steady state for third order EX-RK-LDG and
third order EIN-LDG methods with $100$ mesh cells in $[0,0.6]$.}
\begin{center}\footnotesize
\renewcommand{\arraystretch}{1.3}
\begin{tabular}{|c|c|c|c|c|c|c|}
\hline
& 3rd order EX-RK-LDG  &\multicolumn{5}{c|}{3rd order EIN-LDG}
\\
\hline
$\dt$  &4.604E-6 & 1.2E-4     &1.8E-4 & 2.4E-4   &3.0E-4 & 3.6E-4    \\
\hline
$nt$  &265231 &   13517 & 9253   &7069 & 5735 & 4834\\
\hline
$t$  &1.272 & 1.622& 1.666 & 1.697  &1.720  & 1.740\\
\hline
CPU time  &506 & 59.51& 41.39 & 32.25  &27.27 & 22.99\\
\hline
\end{tabular}
\end{center}
\label{table5}
\end{table}

\begin{table}[htp]
\caption{The time step $\dt$, the number of time steps $nt$,
the numerical steady time $t$,
and the CPU time to reach the steady state for third order EX-RK-LDG and
third order EIN-LDG methods with $200$ mesh cells in $[0,0.6]$.}
\begin{center}\footnotesize
\renewcommand{\arraystretch}{1.3}
\begin{tabular}{|c|c|c|c|c|c|c|}
\hline
& 3rd order EX-RK-LDG  &\multicolumn{5}{c|}{3rd order EIN-LDG}
\\
\hline
$\dt$  &1.151E-6 & 1.2E-4     &1.8E-4 & 2.4E-4   &3.0E-4 & 3.6E-4    \\
\hline
$nt$  &930776 &   13508 & 9248   &7065 &5732  & 4831\\
\hline
$t$  &1.122 & 1.621& 1.665 & 1.696  &1.720  & 1.739\\
\hline
CPU time  &5434.47 & 205.41& 140.76  & 108.22  & 85.27& 72.04\\
\hline
\end{tabular}
\end{center}
\label{table6}
\end{table}

\begin{figure}[ht]
\centering
\subfigure[electron concentration $n$ ($10^{12}\rm{cm}^{-3}$)]
{\begin{minipage}[b]{0.45\textwidth}
   \centering
    \includegraphics[width=3in]{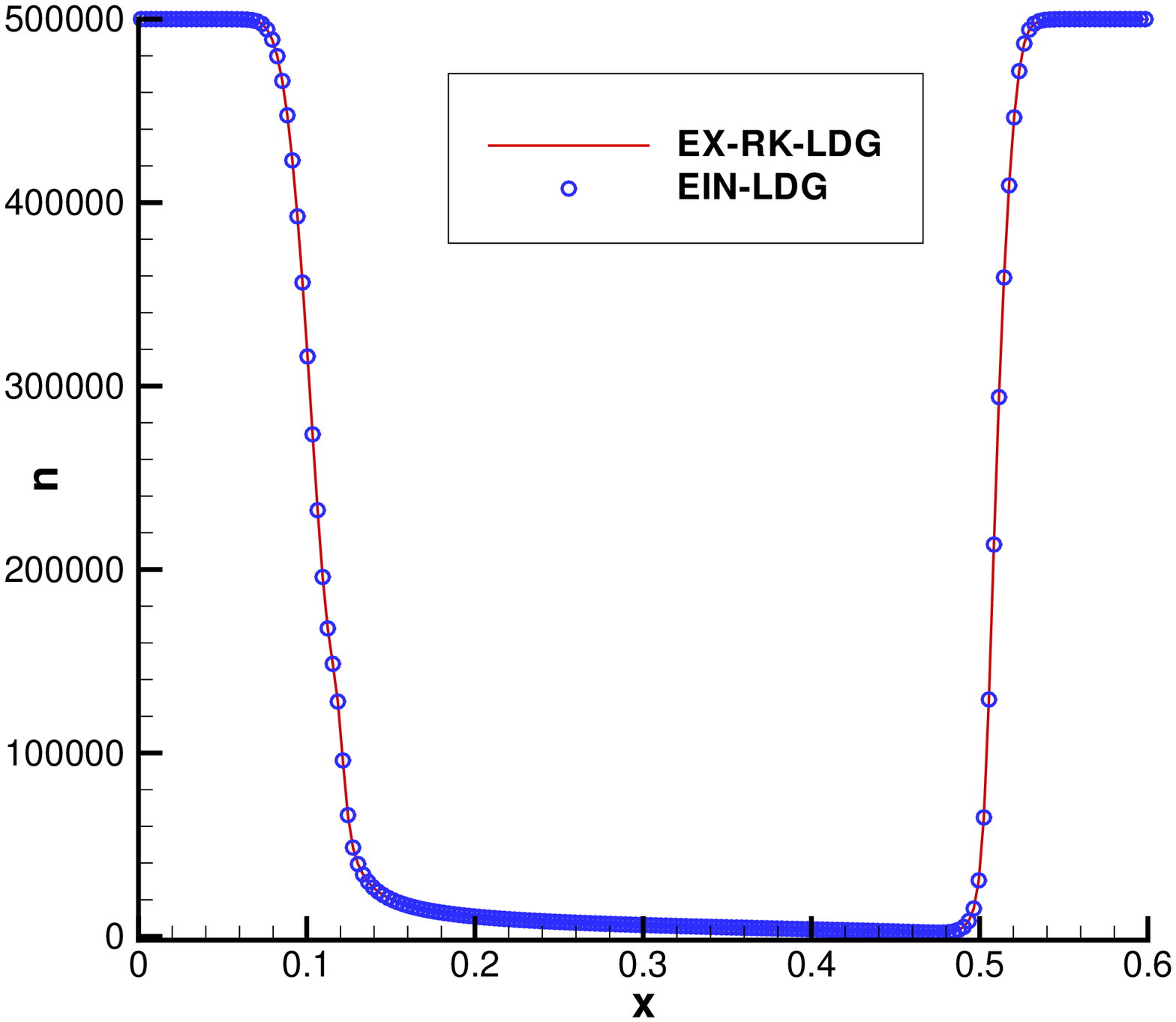}
  \end{minipage}}%
\subfigure[electric field $E$ (V/$\mu$m)]{
\begin{minipage}[b]{0.45\textwidth}
   \centering
    \includegraphics[width=3in]{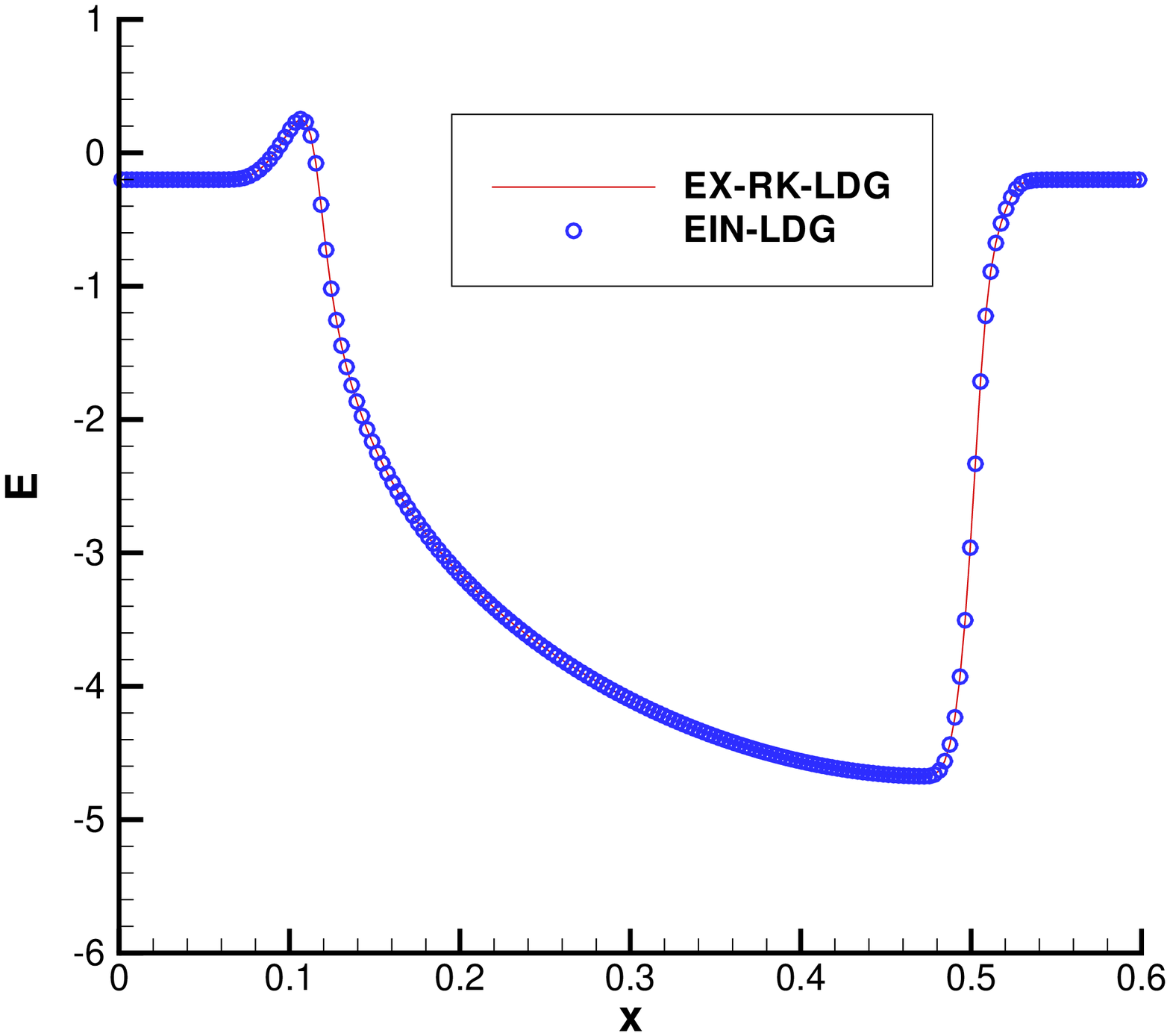}
  \end{minipage}}
\caption{The simulation results of HF model in $[0,0.6]$
with 200 mesh cells, for third order EX-RK-LDG and third order EIN-LDG methods,
$\dt$=3.6E-4 in EIN-LDG method.}
\label{fig5}
\end{figure}

\section{Conclusion}
\label{sec6}
We have developed a class of EIN-LDG schemes for solving one-dimensional
nonlinear diffusion problems,
where a constant diffusion term is added and subtracted to the original equation,
and then one of the terms is treated implicitly and the remaining
terms are treated
explicitly. We have presented the stability and error analysis of the first
and second order EIN-LDG schemes for a simplified  model,
and based on the stability result we have
provided a guidance for the choice of $a_0$
to ensure the unconditional stability of the schemes.
Numerical experiments show that the proposed first and second order schemes are stable and
can achieve optimal orders of accuracy when $a_0 \ge \max\{a(u)\}/2$.
A third order time discretization is also considered numerically.
The schemes have good performance and high efficiency for the PME
and the high-field model in semiconductor device simulations.
The application of the EIN-LDG schemes to solve
two and higher spatial dimensional problems is straightforward,
for which the proposed schemes will be more
efficient compared with explicit or standard
implicit schemes.  This will be left for our future work.

\end{document}